\documentclass[11pt,fullpage]{article}
\usepackage[utf8]{inputenc}
\usepackage{amsmath}
\usepackage{amsfonts,amssymb}
\usepackage{amsthm}
\usepackage{hyperref}
\usepackage{graphics,graphicx}
\usepackage{subfig}
\usepackage{float}
\usepackage{color}
\usepackage[right = 2.5cm, left=2.5cm, top = 2.5cm, bottom =2.5cm]{geometry}
\usepackage{authblk}

\newcommand{\Real}{\mathbb R}
\newcommand{\norm}[1]{\|#1\|}
\newcommand{\abs}[1]{\left\vert#1\right\vert}
\newcommand{\set}[1]{\left\{#1\right\}}

\newcommand{\la}{\lambda}
\newcommand{\N}{\mathcal{N}}
\newcommand{\be}{\begin{equation}}
\newcommand{\ee}{\end{equation}}
\newcommand{\bea}{\begin{eqnarray}}
\newcommand{\eea}{\end{eqnarray}}

\newcommand{\J}{\mathcal{J}}

\newtheorem{thm}{Theorem}
\newtheorem{remark}{Remark}

\newtheorem{prop}{Proposition}

\begin{document}

\title{A model of riots dynamics: shocks, diffusion and thresholds}

\author[a]{H. Berestycki}
\author[a,b]{J.-P. Nadal}
\author[c]{N. Rodr\'iguez\footnote{corresponding author: nrodriguez@unc.edu}}
\affil[a]{\footnotesize{Ecole des Hautes Etudes en Sciences Sociales and CNRS, Centre d'Analyse et de Math\'ematique Sociales (CAMS, UMR8557), 190 - 198 avenue de France, 75013 Paris, France.}}
\affil[b]{\footnotesize{Ecole Normale Supérieure, CNRS, UPMC and Univ. Paris Diderot, Laboratoire de Physique Statistique (LPS, UMR8550), 24 rue Lhomond, 75231 Paris cedex 05, France}}
\affil[c]{\footnotesize{UNC Chapel Hill, Department of Mathematics,
Phillips Hall, CB$\#$3250, Chapel Hill, NC 27599-3250, USA }}

\maketitle 

\begin{abstract}
 
We introduce and analyze several variants of a system of differential equations which model the dynamics of social outbursts, such as riots.
The systems involve the coupling of an explicit variable representing the intensity of rioting activity and an underlying (implicit) field of social tension.
Our models include the effects of
exogenous and endogenous factors as well as various propagation mechanisms.  From numerical and mathematical analysis
of these models we show that the assumptions made on how different locations influence one another and how the tension in
the system disperses play a major role on the qualitative behavior of bursts of social unrest.
Furthermore, we analyze here various properties of these systems, such as the existence of traveling wave solutions, and formulate
some new open mathematical problems which arise from our work.
\end{abstract}


\section{Introduction}
\label{sec:intro}
\label{sec:intro}
This article proposes a framework for describing the internal dynamics of riots, focusing on self-reinforcement and spatial diffusion mechanisms. The purpose here is not to explore the economics, social or political origins of riots, even less to discuss the legitimacy of any given riot.
The approach in our work bears partial similarities to
a recent literature on the use of mathematics in the analysis of uncivil and criminal activities \cite{Short2008,
Berestycki2010, Mohler2011}, showing for instance that statistical regularities in crime patterns, together with insights from criminology, allow one to anticipate the evolution of such patterns, at least on short time scales.
However, the phenomena and the models are quite different.
We now first review some of the most common traits in riots to motivate some of the key ideas in our  model.

Civil disobedience and riots have been and continue to be means for populations or segments of populations to express their discontent towards their government or to react to certain events or political decisions \cite{Moore1978}.
While these episodes of bursts of social activity are relatively rare compared to other phenomena, such as residential burglaries which essentially never cease,
their current significance remains.  If need be, the recent outbursts of protests and civil disobedience that broke
out in Ferguson, Missouri (US) after the fatal shooting of Michael Brown by a police officer on August 9, 2014 reminded us of the relevance of this
phenomenon \cite{Clarke2014}.  Brown, who was African-American, was short by Darren Wilson, a white officer, without apparent probable cause and this
sparked a national and international
conversation about the inadequate and tense relationship between the law enforcement and the community \cite{Lowery}.
A second shooting in St$.$ Louis city reignited the protests and increased the tension between the African American community and the police
force in Ferguson \cite{Salter2014}.

Whether the tension between the population and their ruling government or the police arises from political events or decisions ({\it e.g.} the riots following the assassination of Julius Cesar in 44 BC in Rome, or the New York draft riots in 1863), new or increased taxes (such as the Moscow salt riots in 1648) , food scarcity \cite{Walton1994},
high unemployment \cite{Flamm2005}, police brutality \cite{Lowery, Baudains2012}, or racial tension, these events of social unrest are normally believed to
have been triggered by a specific individual event.
However, one can think of this ``triggering event" as the the straw that broke the camel's back.

  Consider, for example, the beating of Rodney G. King on March 3, 1991
by a group of policemen in Los Angeles, California (US) that was caught on video-tape and was made available for the world to see.  Interestingly enough, it was not the
beating of King that sparked the
riots, but rather the injustice believed to be committed when the police officers involved were exonerated. Indeed, on April 29 of the next year the four police officers involved
in this incident were acquitted and the first incident of the 1992 Los Angeles riots was reported only two hours later \cite{Delk1995,Mucchielli2009}.

Of particular interest to us, is the case of the 2005 riots in France. The triggering event was the incident involving three young men who jumped into a power substation while being pursued by the police
in Clichy-sous-Bois, one of the poorest suburbs of
Paris, on October 27, 2005.
Two of these young men died and this was the spark for the riots that spread throughout the country and lasted over three weeks \cite{Snow2007}.
As a final example, we mention the case of Mark Duggan who was shot on the chest by the police in Northern London on August 4, 2011, this triggered a four-day
riot that spread throughout London \cite{Baudains2012}.

These are only a few examples, but similar episodes continue to be observed throughout the world.  Needless to say that the concept of civil disobedience is not new and has lead to many revolutions
\cite{Arendt1972}.  As Henry David Thoreau put it in his essay {\it Civil Disobedience} the idea behind civil disobedience
is the belief that a just person must stand for what is right, which is not necessarily what is lawful...
``{\it If a thousand men were not to pay their tax bills this year, that would not be a violent
and bloody measure, as it would be to pay them, and enable the State to commit violence and shed innocent blood. This is, in fact, the definition of a peaceable revolution,
 if any such is possible.}"

While it is widely accepted that all episodes of civil disobedience can be traced back to a single event, it is unclear what events are
going to generate a cascade of civil unrest.  In fact, there are many incidents that are very similar in nature to the triggering
events mentioned above which do not generate bursts of rioting activity.
For example, many incidents of shootings by police officers do not generate riots.  If a community is content (employed,
fed, educated, etc.) such incidents might be viewed as unfortunate occurrences, but citizens might not feel compelled to take to the streets
and protest.  However, if the social tension is sufficiently high, the shooting that might soon be forgotten in a content community,
will have a high probability of igniting riots.   In most serious riots in France, the triggering event was the death of a young individual in a poor neighborhood in situations where
the police, or some other official authority, was involved.

From the above examples we see that for a riot outburst to occur the social system needs to be ``ripe," in the sense that the social tension in the system needs to be sufficiently high.
The social tension is a function of the economy, the police relationships with the community, the education
level, as well as individual events.  For example the visit of President Sarkozy to Clichy-sous-Bois on June 2005 where he stated that ``we will clean up the city with a karcher,''
likely raised the tension in this poor community.
Also, in the recent years, solidarities based on religious issues in poor neighborhood have started to play a role in generating local riots in France (as in Trappes, a suburb of Paris, July 2013).

From our perspective, episodes of civil unrest require three factors: exogenous events (these include the triggering event, events
that increase the tension in the system, and even ``pre-triggering'' events), endogenous factors (self-reinforcement in the system), and a
sufficiently high social tension (a ``ripe" system that is ready to experience self-reinforcement).
Endogenous factors are those that are internal to a system, such as would be word-of-mouth type effects that
promote the propagation and organization of the civil unrest.  For example, during the French riots in 2005 the police announced on a
daily basis the number of rioting events that occurred and the whole country was aware of the spread and level of
the rioting activity.
 As a matter of fact, some believed that these announcements, if anything, were  fueling the continuation of the riots.  However, the government stopped announcing these numbers three days prior to the total cease of rioting activity and at this point the level of
rioting was already quite small (less than a total of one hundred events per day).
On the other hand, exogenous factors are external factors to the system that also affect the behavior of a system but in a different way.

This dichotomy has been observed in various systems, such as world-wide-web searches and the number of times YouTube videos are viewed.   Recently, in \cite{Mohler2011} Mohler and collaborators introduced the idea of modeling certain criminal activity, which experience
repeat and near-repeat victimization \cite{Short2009}, as an epidemic-like phenomena using Hawke's processes.
Crane and Sornette in \cite{Crane2008} introduced
a method to determine the quality of highly viewed YouTube videos by using a Hawke's process-like model to extract the effects of the exogenous
factors versus the endogenous factors.  As noted in \cite{Crane2008} a great example of endogenous factors playing a significant role
is the number of views of the Harry Potter trailer; on the other hand, the ``tsunami" keyword search outburst was completely
generated by the exogenous factor, mainly the tsunami that shook Japan in 2011.  One of the advantages of the model introduced in
\cite{Crane2008} is that it affords the ability to extract the quality of videos.  For example, a video with a sharp increase
in the number of views and sudden decrease is more likely to have been boosted by an exogenous factor, but the quality of the video is
probably not sufficiently good for viewers to pass along to others in their social network.  On the other hand, the number of views of a video
that experiences a slower but steadier increase is more likely to have been fueled by the quality of the video.
An analogy can be drawn to civil unrest: episodes with a sudden spike and rapid self-relaxation are likely the mark of
a strong exogenous factor and episodes which experience a slower but steady increase are likely to be fueled by the endogenous factors.
The latter occurs when the system is experiencing a high and slowly decaying tension, and one can argue that these are more serious episodes or,
at least, it is natural to expect that they will last longer.

When dealing with systems that are inherently spatial a fourth factor is the influence that one location has on another.
This effect has become extremely
important due to the globalization of information brought about by the spread of technology
and social media, a particularly suggestive example of the importance of this effect was observed in the ``Arab Spring", a revolutionary wave of riots, demonstrations,
and protests that began in 2010 and spread throughout many countries including Tunisia (where the triggering event occurred), Egypt, Syria, and Libya - see \cite{Lynch12}
for a historical account and \cite{Lang2014} for a mathematical model related to revolutions.
Therefore, it is of much interest to understand how the rioting activities spread spatially.  What
leads some riots to spread while others remain localized?  In France, for example, there has been riots before
and after 2005, but these riots remained at a local level - as in Vaulx-en-Velin, suburb of Lyon, October 1992, the {\em Sapins} neighborhood in Rouen in January 1994, the {\em La Duch\`ere} neighborhood in Lyon, October 1995, etc. In October 1995, in Vaulx-en-Velin again, a rather severe riot of about two hundred young people extended to the rest of the suburb of Lyon, but did  not lead to riots on a  national scale. Similarly, in 2007 at Villiers-le-Bel, a two day
 riot propagated only to neighboring cities. A sociologist called these riots ``les émeutes de la mort'' \cite{Peralva}.

Early models of riots formation have focused on the emergence of a collective phenomena due to herding behaviors \cite{Schelling1973,Granovetter1978}.
Such models, formally related to Ising models in physics, have led to an important literature with various applications in social
and economic sciences - for recent works, see \cite{Gordon2009,Bouchaud2012} and references therein. In the recent years, the generic dynamics of riots and other social phenomena,
as evoked above, has attracted more attention. In particular, several works have developed models based on the assumption that these systems are driven by Hawke's processes \cite{Hawekes1971} -- see for
example \cite{Ogata1998a,Crane2008,Li2014} and references within.

In this work we propose a model for the bursts of civil disobedience that includes the four mechanisms discussed above: endogenous factors, exogenous factors, sufficiently
high social tension and influences,
both local and non-local.
In the bigger scope one of the ultimate objectives is to extract the strength of the exogenous factors versus the endogenous factors.

As a first step, in
this work we introduce a stochastic system on a network, which seems fitting for this application.
Numerical realizations of the model illustrate a rich set of behavior of the system.
We observe, especially in some parameter regimes, that the model behaves in a qualitatively similar way to what is observed in many
real-world riots - capturing the global behavior without capturing the details.
  Moreover, we explore the effects that the non-local spread of information has on the spread of civil unrest.
Through the development of the model we observe that it is necessary that the social-tension in the system also spreads in order for the system to
experience a large scale burst of activity.  To explore this phenomena mathematically, we then derive a system of nonlinear partial differential equations.

This work presents a general model about riots and was initially motivated by the 2005 French riots. However, we believe that the family of models introduced here  can be applied or adapted to many systems that experience bursts of activity
followed by a period of relaxation.
The general spirit of the model is to combine an explicit function which is observable with an implicit field, here the social tension. The explicit function - here the level of rioting activity, however it is defined - corresponds to actions that can be measured.
 The latter implicit field can be thought of as a potential field.  The plausible existence of an underlying social field in collective social phenomena is stressed in \cite{Bouchaud2013}.
We believe that this approach
is relevant for a number of other situations where the introduction of such an underlying implicit field is warranted.
For example, conflicts of various nature, with their escalation parts present some of the aspects we have described here.  At the same time, models in the same vein but with bistable non-linearity in the
implicit field are relevant to describe the loss of confidence, be it among people, towards organizations, media or towards the state of the economy, is such an instance where a buildup of distrust can enable a seemingly minor action or event to precipitate a complete loss of confidence.

 We also note that the system we introduce here has some similarities with models in neuroscience. Specifically it is related to models of neural dynamics that take into account synaptic depression \cite{Tsodyks1998}. We give more details about this link in the next section (see discussion following equation (\ref{sys:homo})). It would be interesting to further explore these analogies.

 We wish to emphasize that our objective here is to introduce simple models whose solutions exhibit the `stylized facts' observed in riots,
which can vary significantly from riot to riot.  For example, the 2011 London riots was on the rise for four days and essentially ceased abruptly on day five,
whereas the 2005 French riots took a period of about twenty-five days with long periods of increased activity and self-relaxation.
There have been models for riots introduced in the literature previously -- see for example \cite{Braha2012, Lang2014, Davies2013}, which we discuss in more detail below.
However, to the authors' knowledge this is the first PDE model developed for riots, although the use of PDE systems to model urban crime has recently become an active field of research
\cite{Short2008, Berestycki2010}.  A continuous model affords us the ability
to prove rigorous spreading and decay estimates using PDE methods.
It would be worthwhile, in future work, to move from a model describing the global behavior
to one that actually captures more details about the actual spread of a given rioting activity.

Lastly, we discuss some previous works that are especially related to our research.  Theories on the role of contagion or diffusion processes in collective actions have been proposed already a long time ago (see e. g. LeBon 1895 \cite{LeBon1895}, and references in \cite{Myers2000}), but formal mathematical models are quite recent.
In \cite{Braha2012}, Braha introduced a non-linear spatial dynamical model for
the global spread of civil unrest that includes short-range connections (describing geographic locations) and long-range connections
(describing the effects of social networks and the media).  In this work the author concludes that external causes such as those mentioned in the introduction (racial
tension, food scarcity, etc.) are not necessary for the sudden outburst of civil disobedience.  This is in contrast to our observation that the dynamics of these
external causes are essential to fully understand any rioting activity.
Davies and
collaborators  studied a model in \cite{Davies2013}, similar in nature to that of Braha, but which included the effects of police deterrence.
This model was particularly concerned with data from the 2011 riots that took place in London.  The authors tune the parameters of the model and obtain
a simulated bursts of rioting activity that is qualitatively similar to what occurred during the 2011 riots.
Finally, we mention a compartmentalized model for the dynamics of revolutions \cite{Lang2014}.  The authors of that work
use a simple differential equation model that includes the effects of police repression and censorship.  Under the assumption
that for a revolution to propagate there needs to be a sufficient amount of protesting taking place, the authors divide the parameters
of the model into regions that would lead to either stable or unstable regimes.

While the models we introduce here contain similar ideas to those used in the works mentioned above, our model has a wider scope.  It contributes the
dynamics of the social tension in the system and a continuous model that allows us to rigorously analyze certain key characteristics
 of the spread of riots.  The use of continuous models, which are less conventional in these contexts,
to describe social phenomena has been popularized in recent years.
Of particular interest are the models of Short and collaborators in \cite{Short2008} and of Berestycki and Nadal in \cite{Berestycki2010}
to describe the propagation of crime.  These works introduce a notion of an invisible scalar field that measures the probability that a criminal activity occurs, this field is refereed to as the
``attractiveness field" in \cite{Short2008} and the ``willingness to commit a crime" in \cite{Berestycki2010}.
In the context of riots or civil unrest, this is analogous to a measure of the social tension.  This concept is also
found in the model of \cite{Braha2012}, which includes a measure of the political, social, and economic stress.  

{\it Outline of the paper.}
In section \ref{sec:discrete-model} we introduce the model on a network and illustrate the results of some numerical
realizations of the system on a single site in section \ref{sec:num-exp}.  We analyze the dynamics in the simplest case of a single site in section \ref{sec:discrete-analysis}.
In the following section we perform and illustrate some numerical experiments on a network.  We derive the continuous system in the case of local and non-local spread of the social tension in section \ref{sec:continuous}.
We discuss the propagation of rioting activity in section \ref{sec:waves}. We conclude with a discussion in section \ref{sec:disc}.

\section{Description of the model}
\label{sec:discrete-model}
It is natural to consider a network of $N$ nodes, where each node represents a location that is prone to rioting activity.  These nodes correspond to the
``urban clusters" of \cite{Braha2012} and they can represent, for example, cities in a country or neighborhoods within a city
that are likely locations for the gatherings of people who are protesting or participating in more violent and destructive activities, such as
arson or looting.  Let us denote this network by $\mathcal{N},$ we discuss the connections between the nodes in a
network shortly.  We assume that for any node $s\in \mathcal{N}$ there is an explicit 
field that measures the {\it level of rioting activity}.
The level of activity is dynamic in time and we denote it by $\lambda(s,t)$.
We further assume that there is a base intensity rate $\lambda_b(s)$ which can vary between nodes on the network.
This base level determines the low recurrent activity that occurs in the absence of any unusual factors.
For example,
according to media sources, the typical number of burnt cars in France is between fifty to one-hundred per night, with as many
as three to four hundred burnt during special
times, such as New Year's eve.
Interestingly, since around 1999 the burning of cars, in the absence of riots, has become a national sport in France.  In fact, it is
estimated that 10-30$\%$ are insurance crimes.  This is clearly a component of the self-reinforcing mechanism:  some people might burn their own cars
after other cars have been burnt in their neighborhood so as to make the insurance company believe that they are victims of these events.

Under the assumption that the number of rioting
activities follows an inhomogeneous Poisson process, the $\lambda(s,t)$ would correspond to the intensity of the process, which is the expected
number of rioting events.   In this case, the expected number of events that occur during the time interval $(a,b)$ at node $s$ is given by:
\[
\lambda(s,(a,b)) =\int_a^b\lambda(s,t)\;dt.
\]

Assuming that the system is ``ripe,'' then the occurrence of a triggering event will spark
a movement that gains momentum and is self-reinforcing.  The
endogenous effect (or {\it self-excitement}) is built in the dynamics of $\lambda(s,t).$  Of course, there is a natural saturation limit, e.g. there
is a maximal number of building and cars that can be destroyed.  Thus, we begin with the following dynamics of the level of rioting activity:
\begin{equation}\label {eq:1}
\frac{d}{dt}\lambda(s,t)= -\omega (\lambda(s,t)-\lambda_b(s)) + G(\lambda(s,t)),
\end{equation}
where $\omega$ is the natural mean reverting parameter of the level of rioting activity if there is no self-reinforcing activity.
For simplicity
in  most of this paper we will assume $\lambda_b=0$. However, including a non zero base level of activity will be essential for future validation of the model with data.
The self-reinforcement mechanism is modeled
by the function $G(z)$ which satisfies:
\begin{equation}\label{def:G}
G(z)> 0\;\text{for}\;z\in(0, z_0), \;G(0)= 0\;\text{and}\; G(z)\le 0\;\text{for} \;z\geq z_0.
\end{equation}
 An example of this is a KPP-type term: $G(z) = z (z_0 - z)$ for $z\in (0, z_0)$ for some $z_0>0$.

Since riots are bursts of social activity triggered by exogenous events, but not all external events (similar in nature) lead to riots,
we assume that the systems must be ``sufficiently ripe."  To express this mathematically, we introduce an implicit variable that represents
the readiness of the system to experience these bursts.  We refer to this scalar field as the {\it social tension} and denote it by $\alpha(s,t).$
Naturally, this value can vary from cluster to cluster and it is dynamic in time.
Moreover, we assume that external events, such as, controversial political remarks, lack of social justice, high
unemployment rates, police brutality, and the triggering events, tend to increase the tension in the system.  In a way the social tension measures the
level of resentment that a community or population feels toward the authority they are facing.  Thus,
we think of the kindling of riots as being
analogous to flame propagation: the endogenous factors take substantial effect when the tension has reached a {\it critical tension}.
This leads us to update \eqref{eq:1} as follows:
\begin{equation}\label {eq:2}
\frac{d}{dt}\lambda(s,t)= -\omega (\lambda(s,t)-\lambda_b(s)) +r(\alpha(s,t))G(\lambda(s,t)),
\end{equation}
where $r(z)$ is, for example, a sigmoid function:
\[
r(z)= \frac{1}{1 + e^{-\beta(z-a)}},\]
where $\beta>0$ provides a measure of the transition slope between a {\it relaxed state} (non-excited state) and an {\it excited state}.  In other words,
it provides a measure of how fast the transition is between a system that does not include the endogenous factors and
a system with the full-force of these factors.  The critical tension is denoted by $a$.
Note, that in the limit as $\beta$
approaches infinity $r(\alpha)$ approaches a step function: then as soon as the tension is above the critical threshold
the endogenous factors are in full-force.  We refer to Figure \ref{fig:r} for an illustration of this transition function.
The idea of
a critical threshold was introduced in \cite{Lang2014}, where the authors make the assumption that
the number of protests has to be sufficiently large before it begins to grow into a revolution.  This assumptions leads to a bistable
ordinary differential equation.  However, we note that the critical value is assumed directly in the level of rioting activity, which stands in contrast with
our model.

It is clear that the dynamics of $\alpha(s,t)$ in the process are crucial. Based on the above discussion we first assume that the
exogenous factors increase the social tension in the system.  It is noteworthy to mention that there are numerous factors in a society,
which are present at all times and do not
correspond to a single event, but that also significantly affect the level of tension that a community experiences.  These factors include,
but are not limited to, the state of the economy, unemployment rates, and political tensions.  We include the exogenous factors that can be pin-pointed to
a particular time and place as ``point sources" in our model.  These occurrences are deterministic
and should be clear indicators that the social tension will increase.
In general, if $n$
exogenous events occur at the times and the locations $\left\{s_i,t_i\right\}_i^n\subset \mathcal{N}\times \left\{t>0\right\},$ the source term produced by
this effect is given by $\sum_{i=1}^nA_i \delta_{t=t_i,s=s_i}$, where $A_i$ measures the intensity of the exogenous events and $\delta_{t=t_i,s=s_i}$
is the Dirac delta centered at $(t_i,s_i)$.
To include the more constant factors, such as the state of the economy, we introduce a source term $\alpha_b(s).$  For example,
if the economy is on a downward turn, the social tension will increase: $\alpha_b(s)>0$ during this period.  The second important hypothesis is that, 
in absence of exogenous inputs, the social tension tends to decay (hence $r(\alpha)$ will decay). This hypothesis is somewhat analogous to Myers' proposal that the riot ``infectiousness'' decays gradually over time 
\cite{Myers2000}, which he tests with an econometric approach in the case of the US racial riots in the 60s.
Incorporating these effects into the model gives:
\[
\frac{d}{dt}\alpha(s,t) =  \sum_{i=1}^nA_i\delta_{t=t_i,s=s_i} -   h(\lambda) \alpha(s,t)+\theta \alpha_b(s).
\]

The function $h(\lambda)$ represents the effect that riots or protests can have on the tension: the higher the level of rioting activity the
slower the tension decays.  For example, $h(\lambda)$ could have the form
\begin{equation}\label{def:h}
h(\lambda) = \theta \exp (-p \;\lambda)\quad\text{or} \quad h(\lambda)= \theta \, \left(1+\frac{\lambda}{\lambda_1}\right)^{-p}  
\end{equation}
The parameter $\theta=h(0)$ measures the natural decrease of the tension per unit of time, which sets the natural timescale over
which the exogenous factors have an effect.
The equation is set in such a way that in the absence of shocks ($A_i=0$) and when there is no rioting
(i.e. $\lambda=0$), then, the social tension field $\alpha$ reverts to the base rate $\alpha_b$. Indeed, the equation then reduces to $\dot{\alpha}= - \theta ( \alpha - \alpha_b)$.

The parameters $p$ and $\lambda_1$ control the influence that $\lambda$ has on the decay of $\alpha$. In the following, we will set $\lambda_1=1$ and assume a slow decay, taking $h(\lambda)=\theta/(1+\lambda)^p  $ with $0 < p\leq 1$.  
See Figure \ref{fig:h} for an illustration of this function with two different values of $p >0$. However, from the point of view of modeling, both the cases $p>0$ and $p<0$ make sense. They describe different situations. We will discuss the case of a fast decay of activity, that is $p<0$, in further work.

The timescale over which the
exogenous factors have an effect can be different from the timescale over which the endogenous factors have an effect.  For example, it took one year for the
Los Angeles riots to begin after the initial release of the video showing the beating of King.  However, the tension had been
building up and the exoneration of the policemen responsible for the beating increased the tension above the {\it critical tension}.  On the other hand, the riots only lasted six days.  However, these six days were very intense and the total number of deaths surpassed
that of any other riot in the United States with the exception of the New York city draft riots of 1863 \cite{NatGeo}.  This leads us to the condition
that $\omega>\theta,$ so that the exogenous effects can be observed over a longer period of time than the endogenous effects.

\begin{figure}[H]
  \center
  \subfloat[Transition function: $r(z)$]{\label{fig:r}\includegraphics[width=0.5\textwidth]{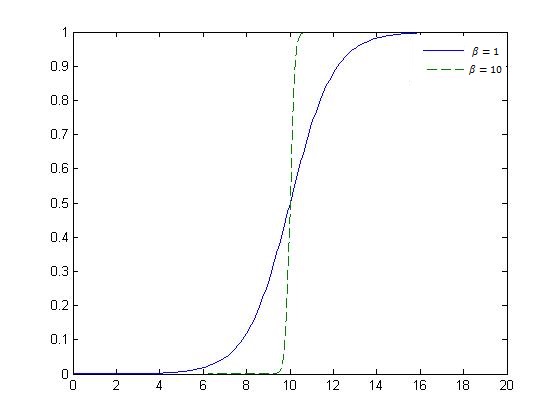}}
 \subfloat[Decay control function: $h(z)$]{\label{fig:h}\includegraphics[width=0.5\textwidth]{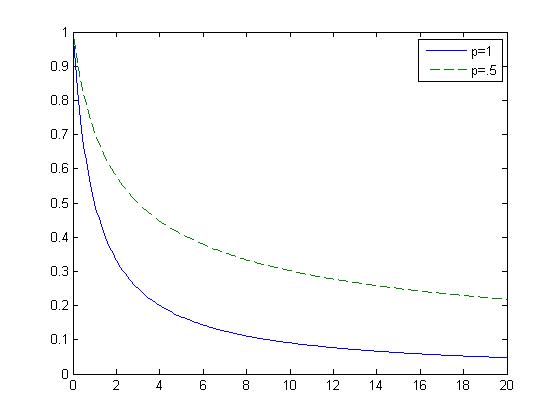}}
  \caption{(a) Illustration of the function $r(z)$ with $\beta=1$ and $\beta = 10.$ (b) Illustration of the function $h(z)$ with $\theta=1$ for $p=1$ and $p=0.5$.}\label{fig:func}
\end{figure}

Including all these elements, the model then takes the form of the following system:
\be\label{sys:homo}
\begin{cases}
\frac{d}{dt}\lambda(s,t)= -\omega (\lambda(s,t)-\lambda_b(s)) +r(\alpha(s,t))G(\lambda(s,t)),\\
\frac{d}{dt}\alpha(s,t)=  \sum_{i=1}^nA_i\delta_{t=t_i,s=s_i} -   h(\lambda) \alpha(s,t)+\theta \alpha_b(s).
\end{cases}
\ee

This system of equations has some similarities with the ones of the neural field equations in the presence of synaptic depressions, even though the non-linearities are not of the same form in the two cases. These equations, derived from the classical Wilson -- Cowan equations \cite{wc}, have the form of a mean field model describing neural dynamics. Tsodyks, Pawelzik and Markram \cite{Tsodyks1998} have introduced an extension of this model that takes into account synaptic depression, that is, the decrease of signal transmission at the synapse when there is a high level of activity. One can think of the rioting activity level here as the analogous of the neural activity of a population of neurons, and the social tension as the analogous of the quantity accounting for synaptic depression. The latter represents the amount of resources available at the synapses for signal transmission. In both cases, the second quantity modulates the reinforcement mechanism in the activity dynamics, and relaxes with a time scale
depending on the activity level. However, in addition to the non-linearities, there are two main differences. First, in the neural context, there is no direct external inputs to the synaptic depression field (hence no shock terms as above). Second, when modeling the spatio-temporal neural dynamics, synaptic depression acts on the influence between different neurons or locations \cite{Tsodyks1998,BressloffKilpatrick2010}, whereas in our case, we assume the social tension to be a field acting specifically on the feedback from the local activity on itself but not on the influence from other locations (see below). In other words, here, the social tension field enters as a modulation of the self-excitatory local dynamics of rioting activity.

In the form (\ref{sys:homo}), the model does not include any effects that one node might have on another.   However, recent popularization of social media
has enabled a kind of globalization of information and it is now rare to find isolated communities.  Indeed,
what happens in one city is quickly known world-wide and therefore it is natural to assume that the endogenous and exogenous effects spread in potentially
non-local ways.  The
role of the use of television networks, cell-phones, and social media, to name a few, is the subject of various studies \cite{Alcaide2011,Gonzalez-Bailon2011}.
In particular, the works \cite{Braha2012, Davies2013, Lang2014}, which we discussed in the introduction, either include or note the importance
of the spatial component and diffusion of information due not only to geographical proximity but also to social connections.  Furthermore, in \cite{Baudains2012}
the authors analyze the data from the 2011 London riots and conclude that there is evidence for the diffusion and clustering of rioting activities, which
is a call for spatial models that explore such factors.
Different neighborhoods, towns, cities, or even countries are influenced in spatially heterogeneous ways.  For example, Marseilles,
France is reputed to be a self-centered city that is not influenced by the rest of France, whereas rural communities might be highly
influenced by urban centers.
We assume that the social tension
in a node spreads to nodes which are within their social network.  This communications or social connections are encoded in the
matrix $C = (c_{ij})_{n\times n}$ with
\begin{align*}
c_{ij} = \left\{\begin{array}{ll}
1 &\text{if node $i$ has a social or communication connection to node $j$},\\
0 &\text{otherwise}.
\end{array}\right.
\end{align*}
At the same time, it is also natural to account for geographic proximity effects.  One is naturally concerned about what happens
to neighbors more so than what happens globally.  Indeed, it is clearly observed that the level of rioting activity spreads locally from
a place to neighboring districts or cities.
To account for this we assume that the level of rioting activity
diffuses to geographically neighboring locations.  To quantify the geographic proximity let us define the matrix $V = (v_{ij})_{n\times n}$ such that:
\begin{align*}
v_{ij} = \left\{\begin{array}{ll}
1 &\text{if node $i$ is a neighbor of node $j$},\\
0 &\text{otherwise}.
\end{array} \right.
\end{align*}
We denote the degree of a node $s$, {\it i.e.} the number of geographic neighbors (or edges) that node $s$ has, by $d_V(s)$ and the number of
social connections by $d_C(s)$.  Let $\eta$ represent the total influence that the nodes connected to $s$ have on node
$s.$  The actual effect that node $s'$ has on node $s$ is proportional to the total number
of nodes connected to node $s$.  Thus, when we include the influence of neighboring nodes, the
dynamics of the level of rioting activity is defined by:
\begin{align*}
\frac{d}{dt}\lambda(s,t)= \frac{\eta}{d_V(s)}\sum_j v_{sj}\lambda(j,t)  +r (\alpha(s,t)) G(\lambda(s,t)) -\omega (\lambda(s,t)-\lambda_b(s)),
\end{align*}
which can be written using the graph Laplacian $\Delta_g$:
\begin{align}\label{eq:la_nl}
\frac{d}{dt}\lambda(s,t)= \frac{\eta}{d_V(s)}\Delta_g\lambda(s,t)+\kappa\lambda(s,t) +r (\alpha(s,t)) G(\lambda(s,t)),
\end{align}
where $\kappa=\eta-\omega.$
The graph here is associated to neighboring locations.
The choice of the graph Laplacian is the simplest and most convenient for an initial analysis.  However, it must be said that
if $\lambda(x,t)$ represents, for example, the number of people (or fraction of the population) which are protesting or rioting, then more complex topologies
must be included in order to move toward a more realistic model.  In particular, the inclusion of transportation networks will be crucial.  These heterogeneities
 could lead to non-local, fractional, or non-linear diffusion.  Non-linear diffusion arises naturally if one takes into account that the diffusion should really be proportional
 to the level of rioting activity.  Of course, the behavior of a model with the more general diffusion can be significantly different; however,
the analysis of these cases is beyond the scope of this paper.

From equation \eqref{eq:la_nl} we see that the natural decay of the level of criminal activity must be larger than the
total influence of the neighbors of node $s$.  This gives rise to condition that $\omega>\eta$ to guarantee the
eventual decay or rioting activity.
We assume that the influence on the social tension is governed by the
social/communications network (which could be non-local geographically).  Incorporating this into the model gives
the dynamics of the social tension is governed by the following equation:
\begin{align}\label{eq:al_nl}
\frac{d}{dt}\alpha(s,t)= \frac{\eta}{d_C(s)}\sum_j c_{sj}\alpha(j,t) + \sum_{i=1}^n A_i\delta_{t=0,s=s_i} - h(\lambda(s,t)) \alpha(s,t) +\theta \alpha_b(s).
\end{align}
Combining equations \eqref{eq:la_nl} and \eqref{eq:al_nl} yields the final system on the network.  This system
can easily be generalized to include weights on the influence between any two nodes, in which case, we would use the weighted graph Laplacian.
The following table summarizes the
parameters.
{\small
\begin{table}[H]
\begin{center}
    \begin{tabular}{ | l | l | p{5cm} |}
    \hline
    \textcolor{blue}{Parameters}  & \textcolor{blue}{Description} \\ \hline\hline
 	$\omega$ & Decay rate of the rioting activity level $\lambda$.\\ \hline
	$\lambda_b$ & Base rioting activity level.\\ \hline
	$A_i$ & Strength of the shock at time $t_i$ and location $s_i$.  \\ \hline
	$\theta$ & Decay rate of the social tension value $\alpha$.   \\ \hline
	$p$ & Level of influence that $\lambda$ has on the decay of the social tension.\\ \hline
	$\alpha_b$ & Base social tension value.\\ \hline
	$\beta$ & Sharpness of the transition between the relaxed state and excited state.\\ \hline
	$a$ & Critical social tension value.\\ \hline	
	$\eta$ & Strength of the influence of neighboring nodes.\\ \hline
    \end{tabular}
\end{center}
\end{table}}

We can also consider a stochastic version of this model, which takes into account, for example, the effects of the media
and the climate.  Let $X_t$
represent a Brownian or L\'evy process, then we obtain the stochastic version of the model:
\begin{subequations}\label {eq:3}
\begin{align}
&d\lambda(s,t)= \frac{\eta}{d_V(s)}\Delta_g\lambda(s,t)dt+\kappa\lambda(s,t)dt +r (\alpha(s,t)) G(\lambda(s,t)) dt + \sigma \lambda (s,t) dX_t,\\
&d\alpha(s,t)= \frac{\eta}{d_C(s)}\sum_j c_{sj}\alpha(j,t)dt + \sum_{i=1}^n A_i\delta_{t=0,s=s_i} - \left(h(\lambda(s,t)) \alpha(s,t) -\theta \alpha_b(s) \right)dt.
\end{align}
\end{subequations}

System \eqref{eq:3} is a coupled system of stochastic differential equations.

\section{Numerical experiments of the single site model}\label{sec:num-exp}To demonstrate the flexibility of the model to capture the various global behaviors observed in
a variety of real-world riots, we begin with some numerical experiments.
For this purpose, we consider the following model on single node
without any noise and $\alpha_b=\lambda_b= 0$:
\begin{subequations}
\label{sys:homo_1}
\begin{align}
\frac{d}{dt}\lambda(t) &=-\omega \lambda(t) +r(\alpha(t))G(\lambda(t)) ,\label{sys:homo_1a}\\
\frac{d}{dt}\alpha(t)&= \sum_{i=1}^nA_i \delta_{t=t_i} -  h(\lambda(t))\alpha(t). \label{sys:homo_1b}\\
\lambda(0)&=\lambda_0\quad\text{and}\quad\alpha(0) =\alpha_0,\label{sys:homo_1c}
\end{align}
\end{subequations}
with either one or two shocks ($n=1$ or $n=2$).
\begin{enumerate}
\item {\it Slow relaxation of rioting activity}:  The first simulation results in an outburst of rioting activity that relaxes slowly - refer to Figure \ref{fig:slow}.
The outburst was a result of one shock at time $t=0$ that was intense enough to push the social tension above the critical threshold.
This
type of relaxation was observed during the 2005 French riots.
\item {\it Fast relaxation of the rioting activity}: The second simulation results in an outburst of rioting activity that suddenly decreases - refer to Figure \ref{fig:fast}.
This type of relaxation was observed during the 2011 London riots.  The driving parameter here is $\beta$ (all other parameters
where unchanged from the previous simulation),
as the transition in the function $r$ is sharper the self-relaxation (and self-excitation) is also sharper.

\item {\it Delayed outburst of activity:}  The third simulation illustrates the case of two exogenous events: the first event occurring at
time $t=0$ and the second event occurring at $t=12$.  The first external event is not strong enough to lead to an outburst of activity, but
it does increase the tension in the system.  Thus, when the second external event occurs, this drives the social tension above the critical threshold
leading to what we call a delayed burst of activity - refer to Figure \ref{fig:delay}.  This is similar to what happened in the 2001 L.A. riots: the first event would correspond
to the beating of King and the second event to the exoneration of the police officers involved.
\item{\it Two bursts of activity:} The final simulation results in two bursts of rioting activity.  In this case the first shock was strong enough
to lead to a bursts of activity that settled down before a second shock occurred (of smaller intensity), which reignited the activity - see Figure \ref{fig:double}.
This was observed, for example, in the recent
protests in Ferguson, Missouri.

\end{enumerate}
\section{Analysis of the single site model}
\label{sec:discrete-analysis}
We now provide a more rigorous analysis of the system given by \eqref{sys:homo_1} by first analyzing it in the absence of shocks, 
and then considering the case of a single shock $(n=1)$, and finally the case of repeated shocks.

\subsection{Absence of shocks}
\label{sec:disc-noshock}
We begin with the study of the system:
\begin{subequations}
\label{sys:homo_1bis}
\begin{align}
\frac{d}{dt}\la(t) &=\Phi(\la(t),\alpha(t)), \;\; \Phi(\la,\alpha) := -\omega (\la-\la_b)  +r(\alpha)G(\la)  ,\label{sys:homo_1bis_a}\\
\frac{d}{dt}\alpha(t)&= \Psi(\la(t),\alpha(t)), \;\; \Psi(\la,\alpha) := \theta \alpha_b -  \alpha h(\la)  \label{sys:homo_1bis_b}\\
\la(0)\quad&=\lambda_0\quad\text{and}\quad\alpha(0) =\alpha_0, \label{sys:homo_1bis_c}
\end{align}
\end{subequations}
where $\la_b$ is assumed to be small relative to the maximum of the function $G$.

\paragraph{Case with $\la_b=0$ and $\alpha_b=0$.}
A preliminary remark is that the dynamics $d\lambda/dt=\Phi(\la(t),\bar \alpha)$ for a fixed $\bar \alpha$ value, has
as stable fixed point at $\la=0$ for $\omega > r(\bar\alpha) G'(0).$  On the other hand,
for $\omega < r(\bar\alpha) G'(0)$, the stable fixed point is $\la^*(\bar \alpha)>0$ which is the non-zero solution of $\Phi(\la,\bar\alpha)=0$
- for example $\la^*=z_0-\omega/r(\bar \alpha)$ when $G(z)=z(z_0-z)$.
We thus choose the parameters  such that
\be
\label{hyp:z0r0<w<z0}
G'(0) r(0) < \omega < G'(0)\quad\text{and}\quad \lim_{\alpha \rightarrow \infty}r(\alpha)=1,
\ee
so that, at constant $\alpha$, the attractive state is a no-riot state for small tension and a
 non-zero rioting state for large tension.

The dynamics are easily understood by looking at the nullclines of \eqref{sys:homo_1bis} in the plane $(\alpha, \lambda)$. The nullcline associated to $\la$, $\Phi(\la,\alpha)=0$, is  given by
$\la=0$, and, for $\alpha>\alpha_c$ with $r(\alpha_c)=r_c:=\omega/G'(0)$, the function  $\la_1(\alpha)$, which increases from $0$ at $\alpha=\alpha_c$ to its maximum value $\la^*\leq z_0$ as $\alpha\rightarrow \infty$ ($r(\alpha) \rightarrow 1$ in this limit). This is illustrated on Figure \ref{fig:nullcline} for $G=z(z_0-z)$. During the evolution of \eqref{sys:homo_1bis}, $\la$ increases everywhere under this curve $(\alpha,\la_1) $ if $\la > 0$, and decreases elsewhere.

Since $\alpha_b=0$, the nullcline associated to $\alpha$ is simply $\alpha=0$: indeed $d\alpha(t)/dt$ is strictly negative for any $\alpha>0$.
One can then conclude that there is a single attractive fixed point, $\alpha=0,\lambda=0$. If the initial values are in the domain left/above the $\la$-nullcline, this fixed point is reached with continuously decreasing $\alpha$ and $\la$. If one starts below the nullcline, with $\lambda(t=0)>0$, the fixed point is reached after an excursion at high $\lambda$. While $\alpha$ is permanently decreasing, $\lambda$ increases until the nullcline is reached with a null slope $d\la/d\alpha=0$, and then decreases towards $0$. If $\la=0$ and $\alpha>\alpha_c$, the system is at an unstable fixed point. However, if one adds any small perturbation,  the system will be driven into the $\la >0$ domain where the rioting excursion will occur.
For an illustration, see Figure \ref{fig:nullcline} (the $\la-$nullcline is the dashed-red curve).

\begin{figure}
  \begin{center}
  \subfloat[Slow self-relaxation]{\label{fig:slow}\includegraphics[width=0.45\textwidth]{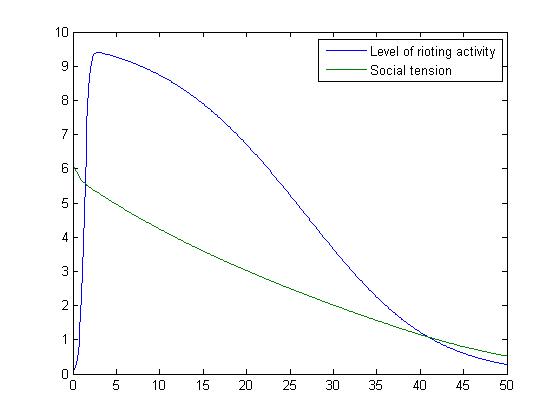}}
  \subfloat[Sharp self-relaxation]{\label{fig:fast}\includegraphics[width=0.45\textwidth]{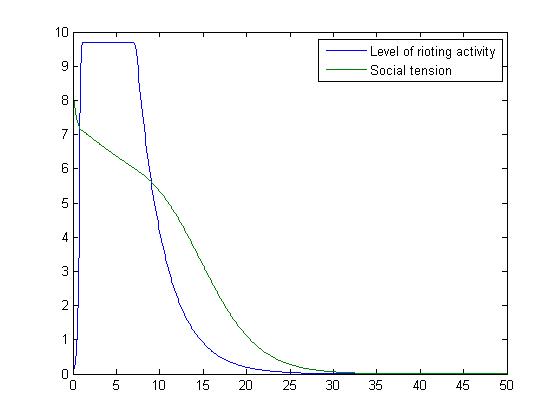}}\\
  \subfloat[Delayed burst of activity]{\label{fig:delay}\includegraphics[width=0.45\textwidth]{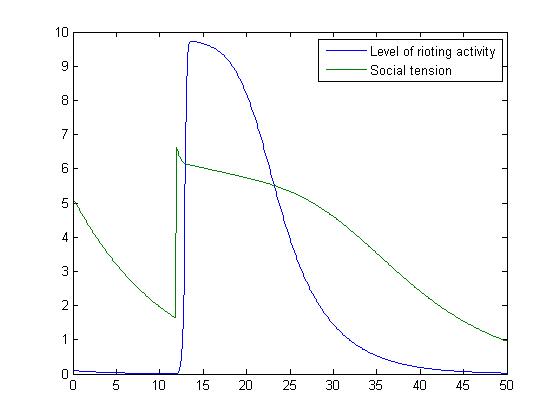}}
  \subfloat[Double burst of activity]{\label{fig:double}\includegraphics[width=0.45\textwidth]{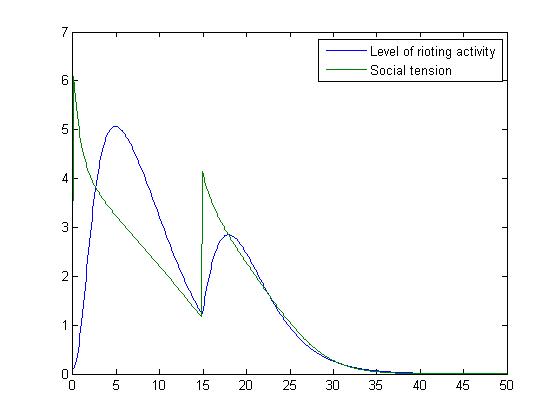}}
 \caption{Simulations of the system given by \ref{sys:homo_1}). The solutions illustrated in \ref{fig:slow} use parameters: $z_0 = 10;\omega = .2; A = 5;\theta=.1$; 
$p=1;\beta=10,a=6, t_1=0, t_2 =12.$  The solutions illustrated in \ref{fig:fast} use parameters:
 $ z_0 = 10;\omega = .2; A = 6;\theta=.1$; 
$p=1;\beta=1,a=6, t_1=0.$  The solutions illustrated in \ref{fig:delay} use parameters: $z_0 = 10; \omega = .3; A = 8;\theta=.3$; 
$p=1, \beta = 100, a = 6.$.  Finally, the solutions illustrated in \ref{fig:double} use parameters: $z_0 = 10;\omega = .3; A_1 = 6;A_2 =3, \theta=.4$; 
$p=1, \beta = 1, a=6$.  
}
  \end{center}
\label{fig:sim_single}
\end{figure}

\vspace*{5pt}\paragraph{Case with $\alpha_b>0$, $\la_b>0$.}

When $\la_b>0$, the fixed point at $\la=0$ is replaced by a fixed point at a small $\la$ value. Along the $\la-$nullcline, as $\alpha$ increases, $\la$ remains almost constant up to $\alpha_c$, and then $\la$ increases sharply as for the case of $\la_b=0$. Because $\la\geq0$, there is no more an unstable fixed point at a small $\la$ value.

For $\alpha_b>0$, the $\alpha-$nullcline is a curve close to the vertical line for small $\alpha_b$, and moves towards the right with smaller slopes. For small $\alpha_b$, the two nullclines intersect in the range where $\la$ is small, and this gives the unique stable fixed point: starting from any initial condition, the trajectory will end at this point, with a riot-excursion if the initial values $\alpha_0, \la_0$ are on the right of the $\la-$nulcline. For a large $\alpha_b$ value, the intersection is in the range where $\la$ is large. The unique fixed point is with a high level of rioting activity.
If $\beta$ is larger than some threshold, there is an intermediate domain of $\alpha$ values for which the two nullclines have three intersects (see Figure \ref{fig:ncB_b}). In that case the intermediate intersect is an unstable fixed point, and the two others are stable fixed points: one with low rioting activity, one with high rioting activity. One has here a structure with a classical discontinuous, `$1^{st}$ order' transition, leading to abrupt changes of behavior and hysteresis phenomena. If $\alpha_b$ increases starting from a very small value, there is a critical value $\alpha_b^1$ at which the fixed point with high $\la$ value appears. As $\alpha_b$ increases, the system may remains on the low fixed point, until a second critical value, $\alpha_b^2$, where the lower fixed point disappears: the system suddenly jumps to the state of high rioting activity. If now $\alpha_b$ decreases smoothly, the system will remain on the high $\lambda$ fixed until $\alpha_b=\alpha_b^1$ where the high fixed point
disappears, and the system is abruptly brought back to the no riot state. This is reminiscent of the ending of the 2011 London riots, which stopped unexpectedly abruptly. The decrease in temperature and arrival of rain has been among the factor contributing to this event. One may assume that the weather contributes to the level of $\alpha_b$.

\begin{remark}
From the Poincar\'e-Bendixon theorem it follows that there is always convergence of the solutions to \eqref{sys:homo_1bis} to a stationary state. Indeed, the 
Poincar\'e-Bendixon theorem states that a trajectory necessarily either converges to a stationary point or converges to a cycle. 
But here, owing to the second equation in which the right-hand side has a sign, there are no cycles. Note that this theorem holds for both cases $p<0$ and $p>0$ it is only the sign that matters. 
\end{remark}

\subsection{Adding a single shock} First consider the case $\la_b=0$ and $\alpha_b=0$. With an initial condition at small $\alpha$ and $\la$, a single shock $A \delta_{t=0}$ on $d\alpha$ sends the system at a higher value of $\alpha$.
If the nullcline is not crossed, both $\alpha$ and $\lambda$ then decrease smoothly towards the fixed point $(0,0)$. If the nullcline is crossed, the trajectory makes a rioting excursion as described above.
If the initial value of $\la$ is zero, a strong shock sends the system onto the unstable fixed point $\la=0$. However, if one adds any small perturbation,  the system will be driven in the $\la >0$ domain where the riot excursion will occur.

In the case when the base level of rioting activity is zero and $\alpha_b = 0$ we observe an eventual
self-relaxation.  The formal proof of eventual relaxation is as follows.
\begin{prop} [Burst of activity with eventual self-relaxation]\label{prop:1}
Let $n=1$ and any $A_1>0.$  If $\la(t)$ and $\alpha(t)$ are solutions to \eqref{sys:homo_1}
 with $G'(0) r(0) < \omega $, and $\la(t=0)>0$, then
\[
\lim_{t\rightarrow \infty} \la(t) = 0\quad\text{and}\quad \lim_{t\rightarrow \infty} \alpha(t) = 0.
\]
\end{prop}
\begin{proof}
We can solve for the social tension explicitly:
\[
\alpha(t) = A_1 e^{-\int_0^t h(\lambda(s))\;ds},
\]
which we then substitute into the equation for $\la$ to obtain:
\[d\la(t)= -\omega \la(t) dt +r\left(A_1 e^{-\int_0^t h(\lambda(s))\;ds)}\right)G(\la(t))  dt.\]
Recalling that $\la^*$ is the maximum value of $\la$, we have the following lower bound
\[
\int_0^t h(\lambda(s))\;ds \ge \frac{t}{1+\la^*},
\]
which allows us to conclude.
\end{proof}
Even though the level of rioting activity eventually ceases (or relaxes to its base level of activity $\lambda_0$) as does the social tension in the system, it is clear that
the higher the intensity of the triggering event the longer the bursts of rioting activity will last.  Indeed, if the shock is strong enough, the maximum value of $\la$ along its trajectory is reached on the nullcline in the domain where the nullcline is in its asymptotic regime $\la \sim \la^*$, so that the activity remains at its maximum for a long period time before being able to decrease - see again Figure  \ref{fig:nullcline} for an illustration.  In the same vein, the parameter $\theta$ plays, in some sense, an even more important role on how fast the solution decays.  Indeed, $\alpha$ decays exponentially fast at a rate
which depends on $\theta$: smaller values of $\theta$ lead to a slower decay.
More formally, we prove that
when the shock is sufficiently large the intensity will remain close to its maximum value
for long periods of time.

\begin{prop}[Long-periods of rioting activity due to a strong shock]\label{prop:2}
Given arbitrary $L>0$ and $\delta>0,$ there exists $A_0=A_0(\delta ,L)$ and $t_o>0$ such that if $A\ge A_o$ then
\[
\la(t) > \la^*-\delta \quad \forall \quad t\in [t_0,t_0+L].
\]
\end{prop}
\begin{proof}
From the  $\alpha$ equation \eqref{sys:homo_1bis_a} we get
\[
\alpha(t) = A e^{- \int_0^t h(\lambda(s))\;ds},
\]
where $A$ can be interpreted as the initial shock. To take into account a pre-existing base rate
$\alpha_b$, we could also consider $A= A'+ \theta \alpha_b$ where $A'$ is the initial shock and the result is the same. Since $h$ is bounded from above, from this formula we infer that there exists some $r>0$ such that $\alpha(t) \geq  A e^{-rt}$ for all $t\geq 0$. Therefore, given $\eta>0, T>0$, there exists $A_0$ such that for all $A\geq A_0$ we get
\[
r(\alpha (t)) \geq 1 - \eta, \quad\text{for all} \quad t\in [0, T].
\]
Then, consider the equation
\[
\dot{y} = - \omega y + (1-\eta) G(y), \quad y(0)= \lambda_0>0.
\]
Clearly, $\lambda(t) \geq y(t)$ for all $t\in[0, T]$. By our assumptions, this system admits a unique globally asymptotically stable equilibrium
$z_\eta >0$. Hence, given $\delta>0$, there exists $t_0>0$ such that $y(t)\geq z_\eta - \delta/2$, for all $t\geq t_0$.
Now, $z_\eta \to \lambda^*$ as $\eta\to 0$ and therefore we can choose $\eta$ sufficiently small so that $z_\eta \geq \lambda^* -\delta/2$. Gathering these properties, we see that given $\delta>0$ and $L>0$, we can choose accordingly
$\eta>0$, then $t_0\geq 0$, and then $T = t_0 +L$ for which we get $A_0$ such that
for all $A\geq A_0$, $\lambda(t)\geq \lambda^* - \delta$ on the interval of time $[t_0, t_0 +L]$. This completes the proof of the proposition.
\end{proof}

Figure \ref{fig:dynamics_B} illustrates the dynamics for $\la_b>0, \alpha_b>0$. The general features described in the previous case remain, with two important aspects. 
First, as discussed previously, for $\alpha_b$ larger than some threshold there is a fixed point at high value of $\la$: obviously a strong enough shock will make the system evolve towards this high activity state. 
Maybe more interesting is the behavior below the transition. Suppose the system is `at rest', that is at its fixed point in the absence of shock. Then any shock will drive the system in the 
domain where $\la(t)$ increases. However, if the shock is small enough, the system remains into the domain where the $\la-$nullcline is almost flat, the rioting activity remains very small. See Figures \ref{fig:dynB_a} and \ref{fig:dynB_b}.

\subsection{Repeated shocks}
If multiple exogenous events occur, their intensity and proximity in time will play a role on whether the system remains excited or eventually relaxes.
With strong enough shocks at a high frequency, the shocks will maintain the system within the domain below the $\la$-nullcline, and since in that regime $\la$ always increases,
the rioting activity will remain close to its maximum value $\la^*$.  With shocks at moderate frequency, a given shock occurs while the system is in the regime where $\la$ is decreasing,
and the shock sends back the system under the nullcline. One has then a (regular or irregular) cyclic activity.  With weak enough stimulations, in amplitude and/or frequency,
the system will eventually relax. For this single site case, whatever the shocks may be, $\la^*$ is the maximum maximorum possible value of $\la$, whenever the initial condition is smaller than $\la^*$ (which is of course the case with the hypothesis of no riots at $t=0$).

\begin{figure}
  \center
  \subfloat[A small and a moderate shock]{\label{fig:nc_a}\includegraphics[width=0.5\textwidth]{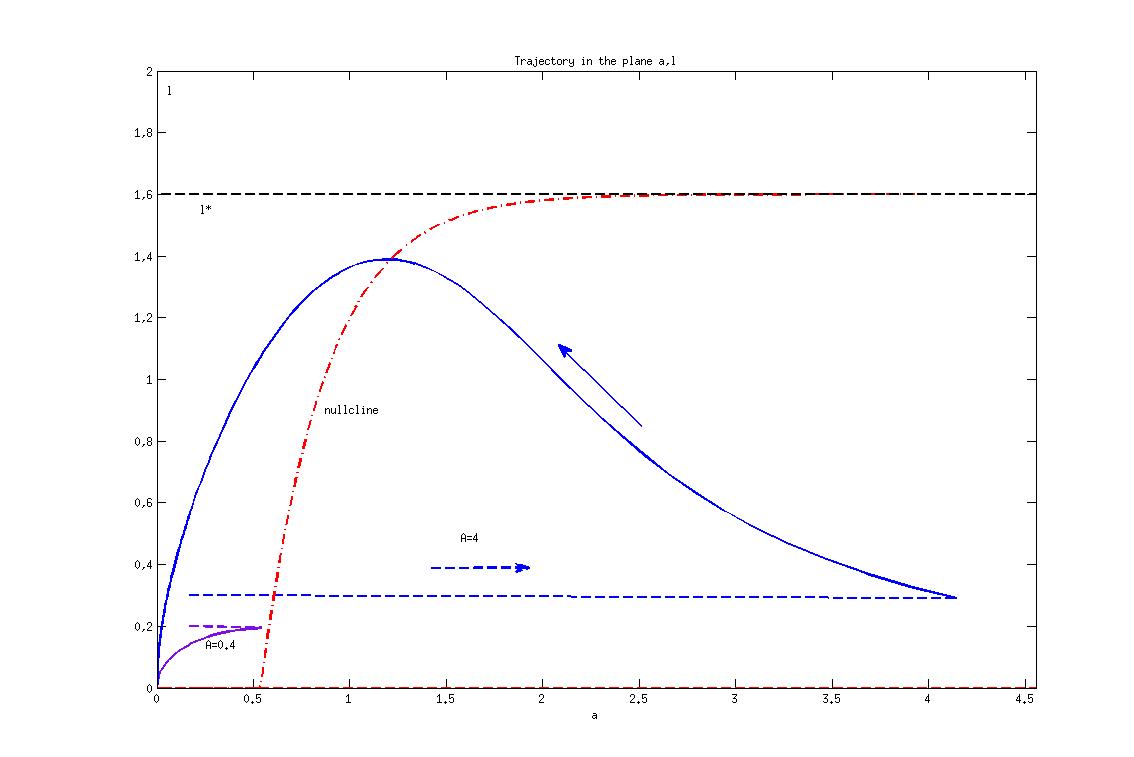}}
  \subfloat[Strong shock]{\label{fig:nc_b}\includegraphics[width=0.5\textwidth]{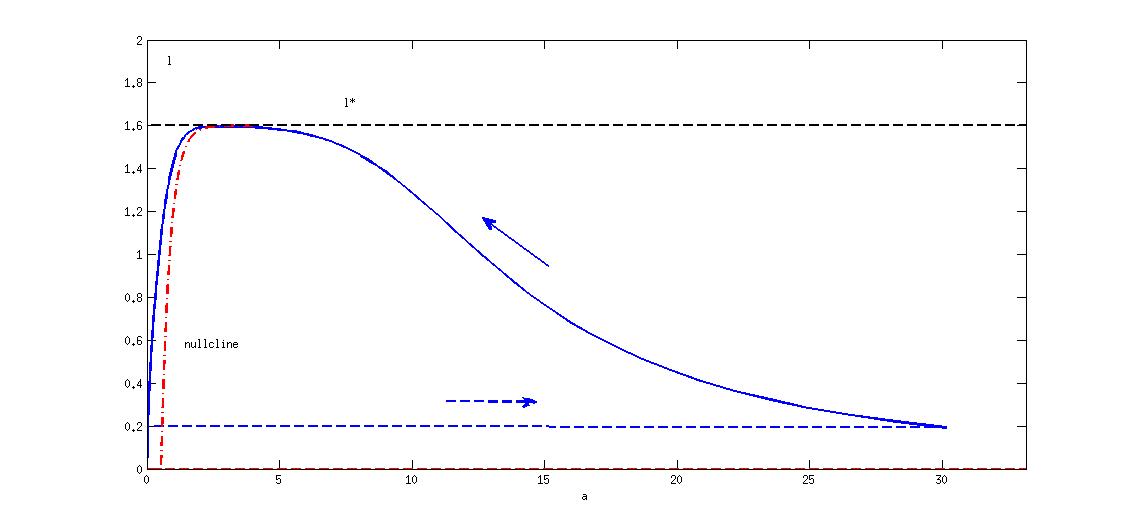}}\\
  \subfloat[Dynamics]{\label{fig:nc_c}\includegraphics[width=0.5\textwidth]{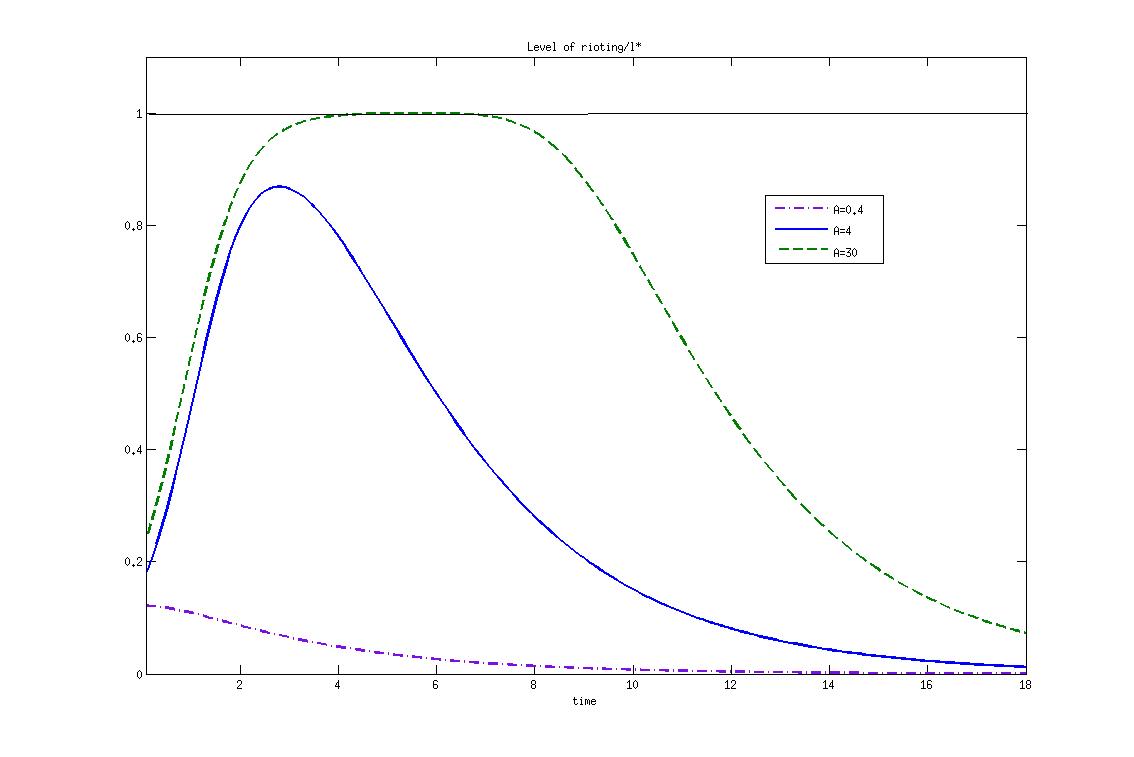}}
 \caption{Simulations of the system given by \eqref{sys:homo_1}
with $\omega = .4, z_0 = 2, \beta = 3, a=1, \theta =.7, p = .7$ and varying intensity of the triggering event: a small shock $A=0.4$, a moderate one $A=4$ and a strong one, $A=30$. Figure \ref{fig:nc_a} shows the nullcline $d\lambda/d\alpha=0$, and the  trajectories for the small (no riot generated) and the moderate (generation of riots) shocks. Figure \ref{fig:nc_b} shows the case of the strong shock (note the difference in the range of $\alpha$ values). In all cases the initial values are chosen with small $\alpha$ and $\la$ values. The dashed lines indicate the jump in $\alpha$ value due to the shock. Figure \ref{fig:nc_c} shows the time dynamics of the rioting activity in the three cases.}\label{fig:nullcline}
\end{figure}

To understand the roles that both the time between events and
the intensities of the shocks play in either leading an eventual relaxation to the base activity value or the perpetual self-excited system,
we explore the case when there are
periodic exogenous events, each with intensity $A$ and period $T$.  It is useful to define the point source term:
\be \label{def:per}
S(t) = A\sum_{i=0}^\infty \delta_{t=iT}.
\ee

Thus, we consider system \eqref{sys:homo_1} with the point source term given by $S(t)$ as defined in \eqref{def:per}.
\begin{prop}[Periodic excitable systems]
Let $\la(t)$ be the solution to system \eqref{sys:homo_1} with point source term given by \eqref{def:per}.
Then, for any given $T$ and for any given $\delta>0$, there exists an $A^*=A^*(T,\delta,\lambda_0)$ such that for all $A\geq A^*$ we have
\[\liminf_{t\rightarrow \infty}\lambda(t) \geq \la^{*}
 - \delta . \]
\end{prop}

\begin{figure}
  \center
  \subfloat[Small $\beta$]{\label{fig:ncB_a}\includegraphics[width=0.45\textwidth]{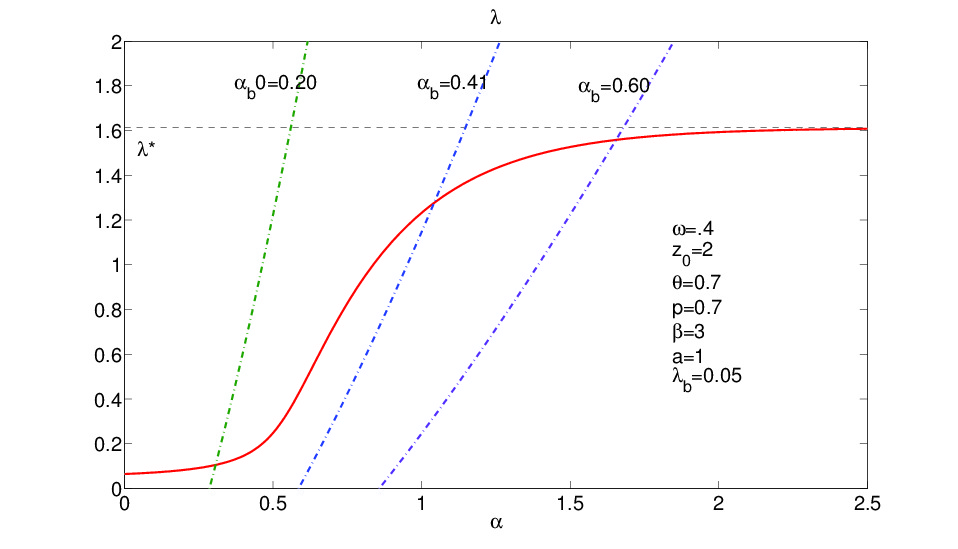}}
  \subfloat[Large $\beta$]{\label{fig:ncB_b}\includegraphics[width=0.45\textwidth]{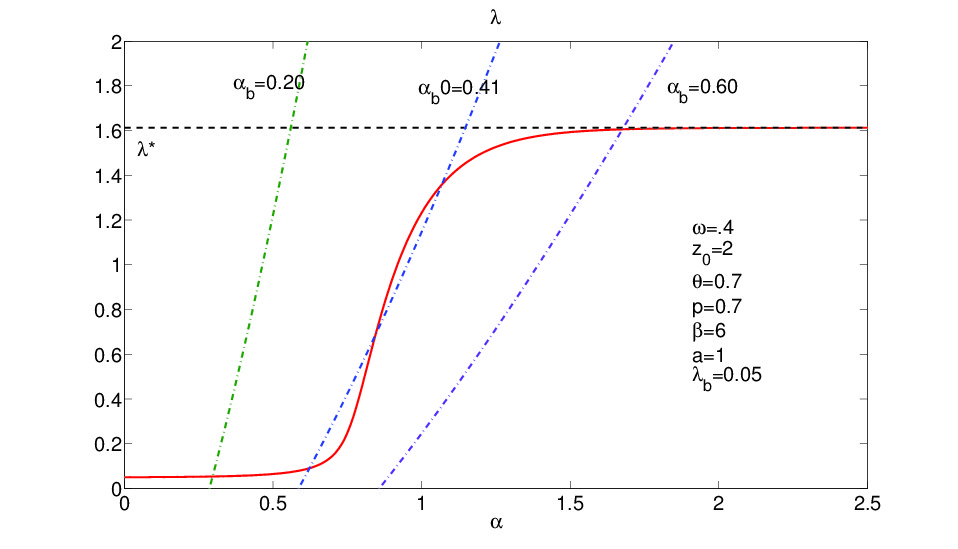}}
 \caption{Nullclines of the system given by \eqref{sys:homo_1bis} with $G(z)=z(z_0-z)$ and $\omega = .4, z_0 = 2,  a=1, \theta =.7, p = .7$, $\la_b=0.05$. Left, \ref{fig:ncB_a}, $\beta = 3$, and right, \ref{fig:ncB_b}, $\beta = 6$. For each case, three $\alpha$-nullclines are shown: $\theta \alpha_b =0.2, 0.41$ and $0.6$. 
 For $\beta=6$, a discontinuous transition occurs. }
\label{fig:nullclines_B}
\end{figure}

\begin{figure}
  \center
  \subfloat[Moderate $\alpha_b$]{\label{fig:dynB_a}\includegraphics[width=0.45\textwidth]{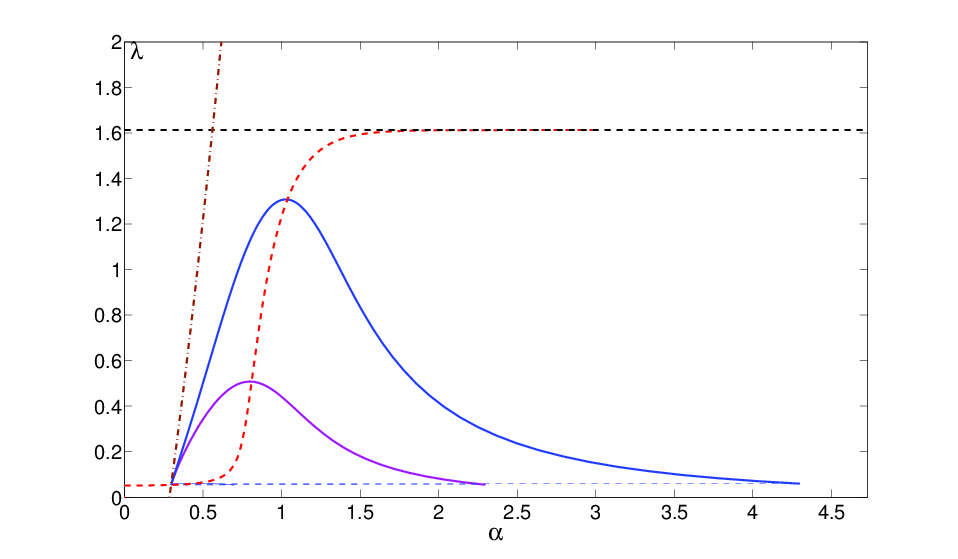}}\quad\quad
  \subfloat[Moderate $\alpha_b$, small shock, zoom]{\label{fig:dynB_b}\includegraphics[width=0.45\textwidth]
   {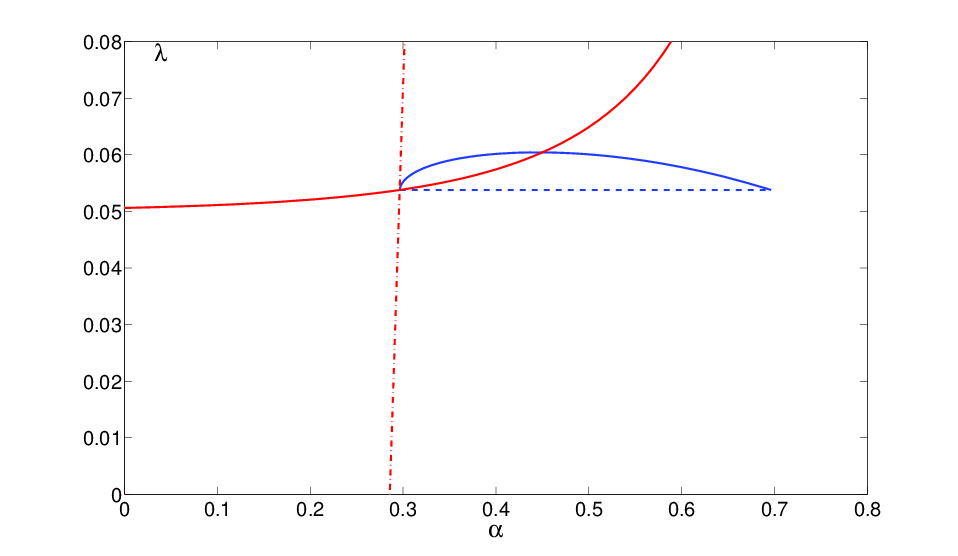}}\\
  \subfloat[Near the transition]{\label{fig:dynB_c}\includegraphics[width=0.45\textwidth]{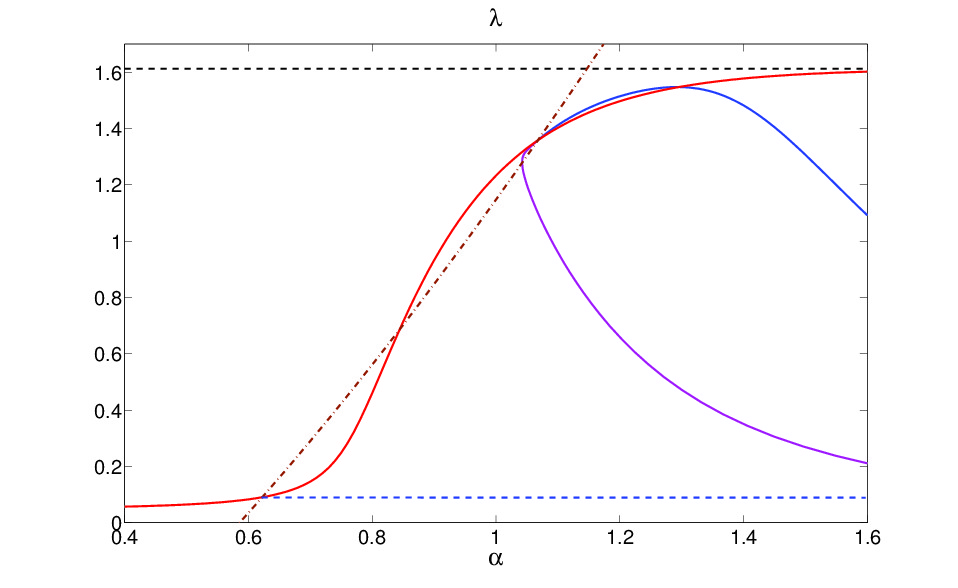}}
 \caption{Dynamics in the phase plane of the system given by \ref{sys:homo_1bis} with $G(z)=z(z_0-z)$ and $\omega = .4, z_0 = 2,  \beta=6, a=1, \theta =.7, p = .7$, $\la_b=0.05$. Top \ref{fig:dynB_a} and \ref{fig:dynB_b}, $\theta \alpha_b = 0.2$. Starting from the fixed point of the system without shocks, three trajectories are shown after stiulating the system with shocks of different amplitudes: $A=0.4, 2$ and $4$. For a small enough shock (\ref{fig:dynB_b}, $A=0.4$), the system remains where the $\la-$nullcline is almost flat, the rioting activity remains very small. Bottom, \ref{fig:dynB_c}, $\theta \alpha_b=0.41$.  In this case, there is two stable fixed points. Under a strong enough shock - here two examples, $A=2$ and $A=4$-, the systems, initially at the low fixed point, ends at the fixed point with a high $\la$ value.}
\label{fig:dynamics_B}
\end{figure}

This is a straightforward consequence of Proposition~\ref{prop:2}.
Figure \ref{fig:periodic} gives a numerical illustration in the case of a periodic excitation as above, starting from a positive value of $\lambda_0$. If the  stimulations are too weak or too
spread out (below some thresholds $A_c(T)$ or $1/T_c(A)$), the activity relaxes to null activity with damped oscillations. Above the threshold, the system converges to a limit cycle.
The proposition states that these limit cycles are in regions closer and closer to $\la=\la^*$ as $A$ gets larger. As shown on Figure \ref{fig:periodic}, this is the same at fixed $A$ when $T$ becomes small (high frequency).

We conjecture a similar result in the case of a  stochastic source term:
\be
S(t) = \sum_{i=0}^\infty A_i \delta_{t=t_i},
\ee
where each $\set{A_i,t_i}$ are specified by a compound Poisson process.
More precisely, the time between events, $t_{i+1}-t_i,$ are specified by $\set{N(t):t\ge 0},$ a Poisson process with a prescribed intensity $\la_P.$ Moreover, $\set{A_i: i\ge 1}$
are independent and identically distributed random variables with a prescribed distribution.  Finally, we assume that
$\set{A_i: i\ge 1}$ are independent of $\set{N(t):t\ge 0}.$

{\it Open questions.}  As $t\to\infty$, we conjecture that there is a limit distribution for $\lambda(t)$.
Now, assume that $\la_b =0$. Then, given $\la_P$ , when the $\mathbb{E}(A_i)$ are bounded and
$\mathbb{E}(A_i) \to 0$, we believe that this distribution converges in probability to 0, and that, when $\mathbb{E}(A_i) \to \infty$, this distribution converges weakly (or in probability) to the constant  $\lambda^*$. It seems to be an interesting question to understand whether there
 is actually a phase transition,  depending on $\mathbb{E}(A_i)$, from a long time limit of $\la_t$ going to zero in probability
to a long time limit of $\la_t$ approaching a non-zero distribution.
This is what we observe in the periodic deterministic case form numerical simulations.
In the random case, we note that whatever the value $\mathbb{E}(A_i)$ is,
we always have
\[
\liminf_{t\rightarrow \infty} \la_t = 0, \quad \limsup_{t\rightarrow \infty} \la_t= \la^{*}
\]
with probability one. Similar questions arise when $\mathbb{E}(A_i)>0$ is given and one examines the influence of the intensity  $\la_P$, which represents here a frequency. We can even assume that the $A_i$'s are constant.
At high frequency, we expect the system to be close to the maximum value most of the time, while for low frequency, we expect it to be close to 0. Whether there is or for what distributions there is a phase transition is an open question.

Some aspects of these questions are supported by numerical simulations. The dynamics with repeated shocks is illustrated on Figure \ref{fig:multirandomtimes} for shocks of constant amplitude $A_i=A$ but occurring at random times, with a mean frequency $\nu=\mathbb{E}(t_i)$. As one would expect, there is a critical value $\nu_c(A)$ such that for $\nu<\nu_c(A)$ the systems eventually relax, and for  $\nu>\nu_c(A)$ the system is in a sustained regime of bursts of excitations. It is observed that the activity is frequently close to the maximum $\la^*$, even at moderate rate. The fraction of time spent near $\la^*$ increases as the mean frequency increases.

\begin{figure}
  \center
  \subfloat[Low frequency, decaying rioting activity]{\label{fig:period_a}\includegraphics[width=0.5\textwidth]{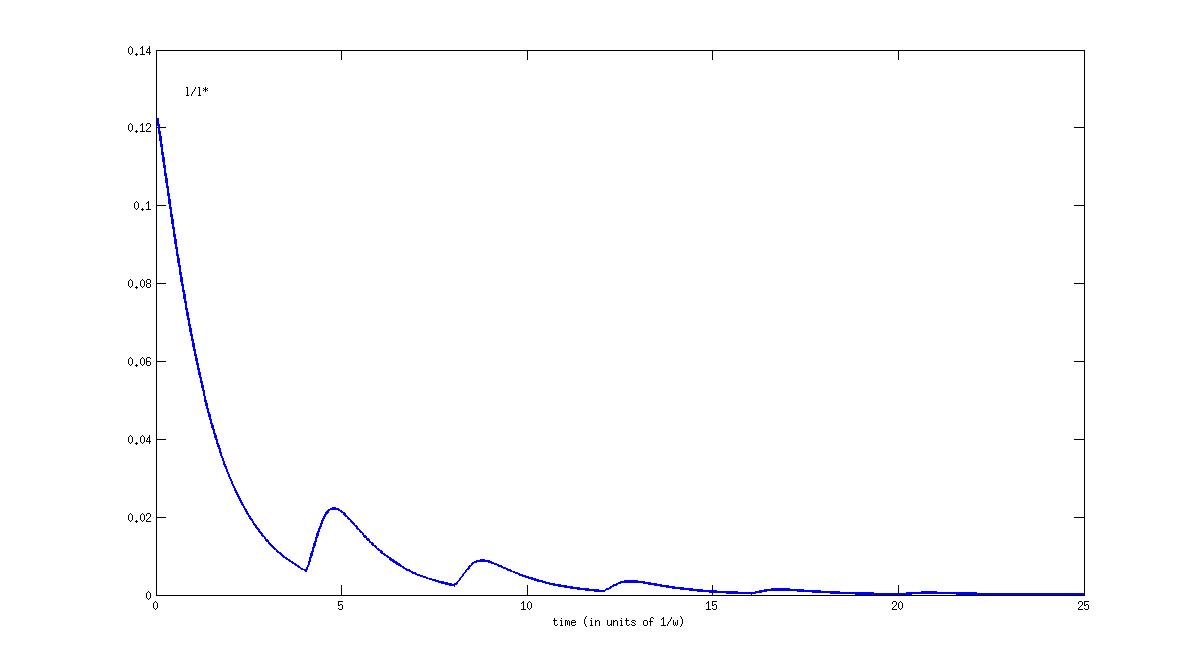}}
  \subfloat[Medium frequency, phase plane]{\label{fig:period_b}\includegraphics[width=0.5\textwidth]{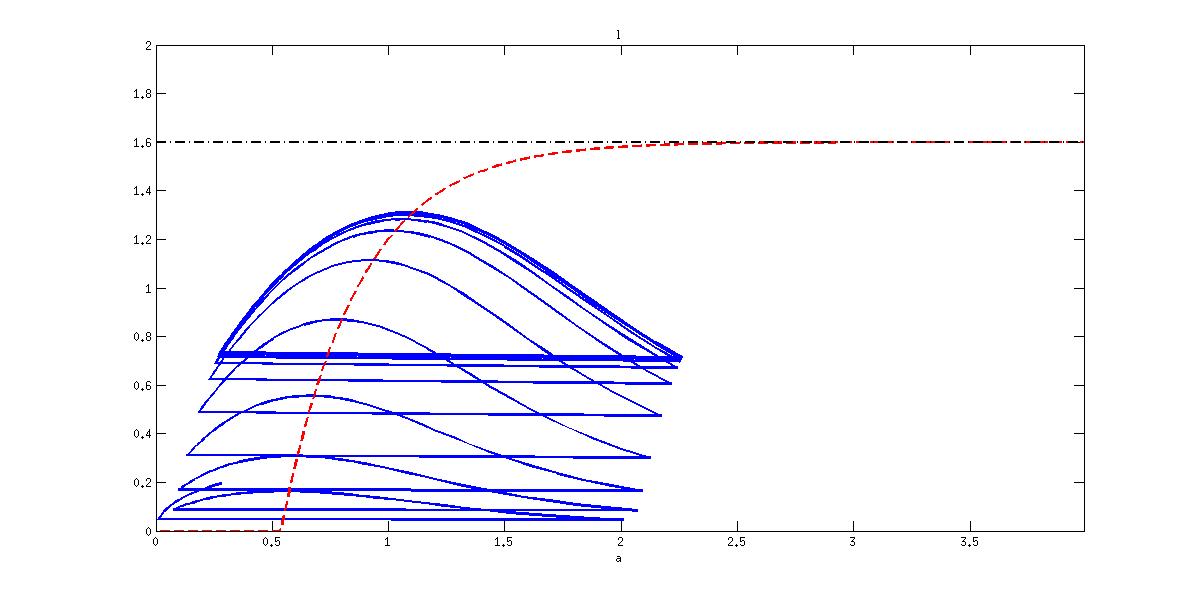}}\\
  \subfloat[High frequency]{\label{fig:period_c}\includegraphics[width=0.5\textwidth]{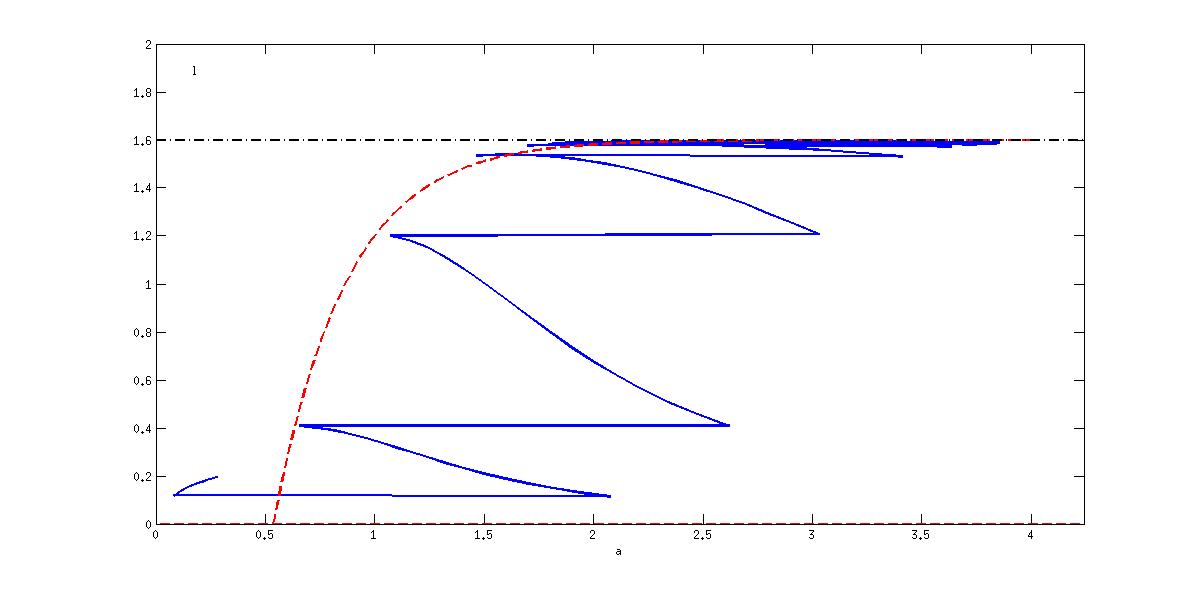}}
    \subfloat[High frequency, zoom]{\label{fig:period_d}\includegraphics[width=0.31\textwidth]{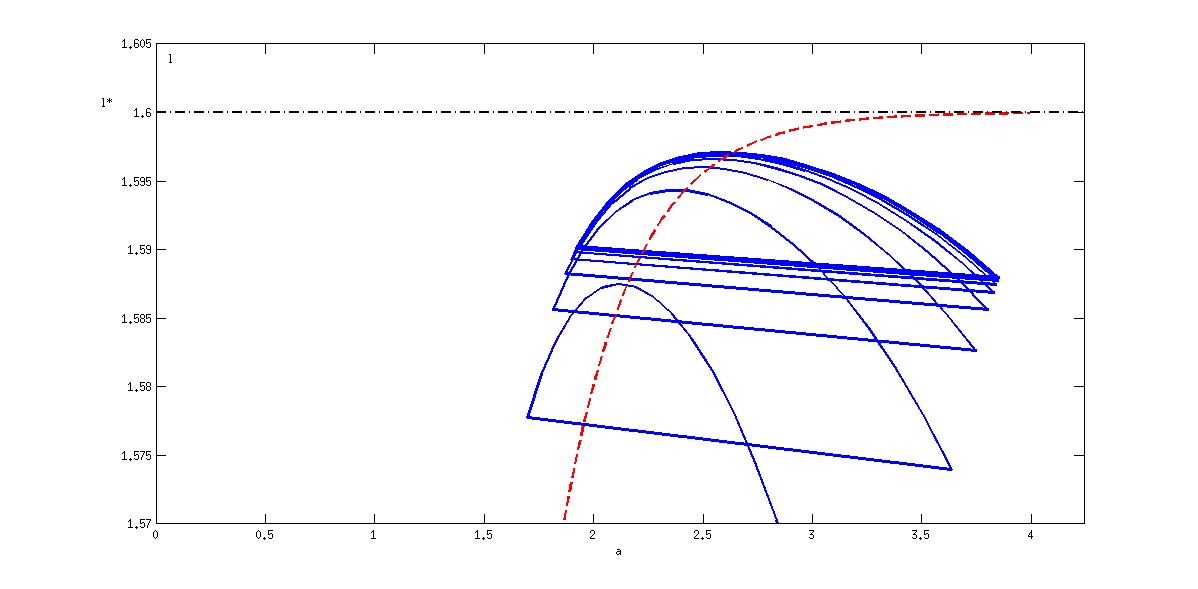}}
 \caption{Simulations of the system given by \eqref{sys:homo_1}
with $\omega = .4, z_0 = 2, \beta = 3, a=1, \theta=0.7, p = .7$ under periodic stimulations at different mean frequencies with common amplitude $A=2$.  }\label{fig:periodic}
\end{figure}
\section{Numerical experiments on a network}
\label{sec:discrete-simul}
To investigate
the effects that spatial dispersal of information has on the spread of riots we illustrate and discuss some numerical realizations of the system
defined in \eqref{eq:la_nl}-\eqref{eq:al_nl} in this section.
We perform simulations on a network of one hundred urban centers that are on a square grid.
The square grid gives the geographic neighbors of each node.

 \begin{figure}
  \center
  \subfloat[Low frequency]{\label{fig:mult_a}\includegraphics[width=0.34\textwidth]{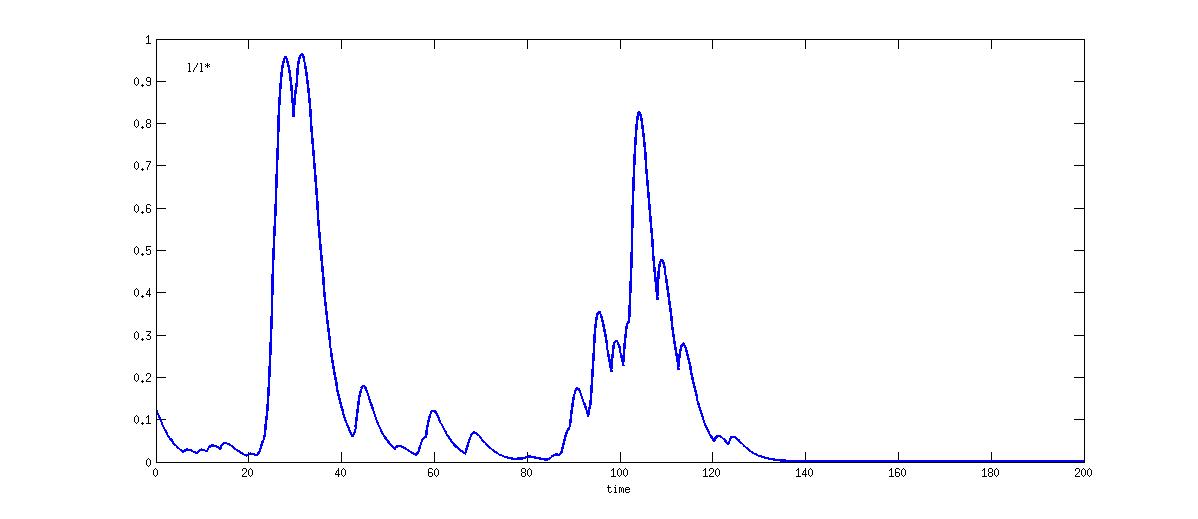}}
  \subfloat[Near the transition]{\label{fig:mult_b}\includegraphics[width=0.34\textwidth]{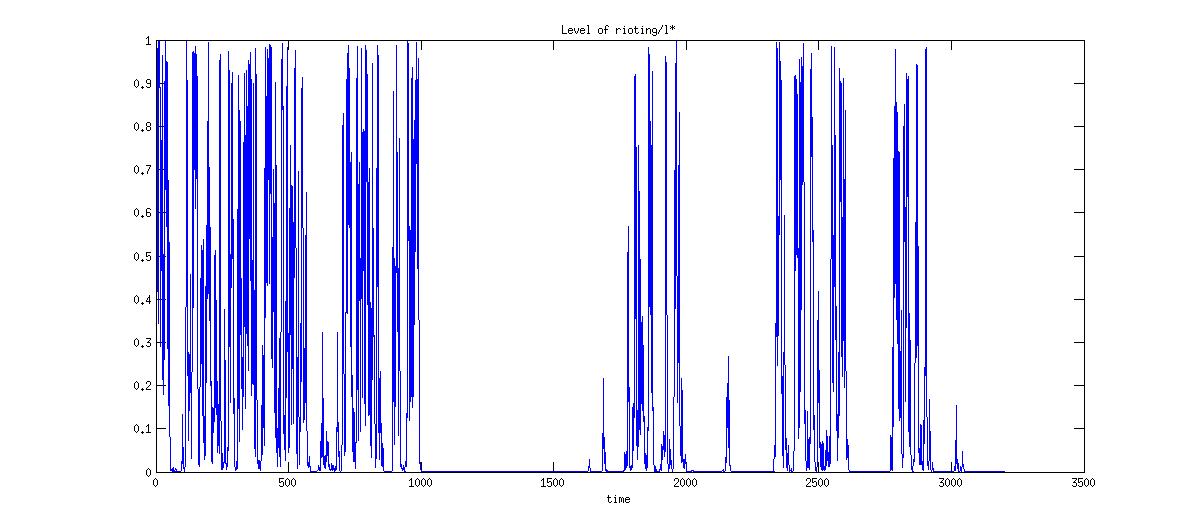}}
  \subfloat[High frequency]{\label{fig:mult_c}\includegraphics[width=0.34\textwidth]{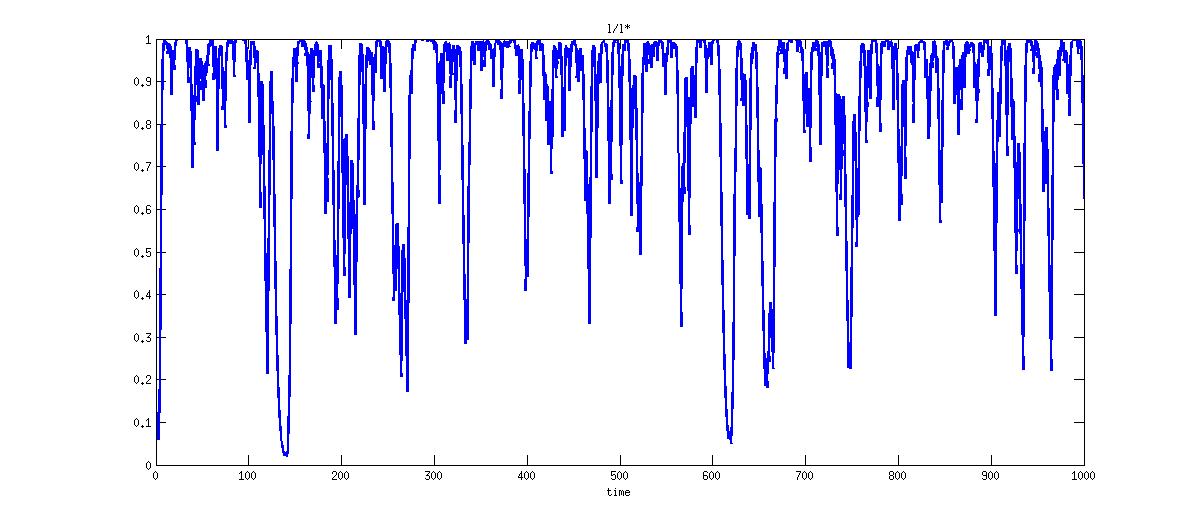}}
 \caption{Simulations of the system given by \eqref{sys:homo_1}
with $\omega = .4, z_0 = 2, \beta = 3, a=1, \theta=0.7, p = 0.7$ under multi shocks stimulations at different mean frequencies with common amplitude $A$. The rioting activity, $\la(t)/\la^*$, is shown as a function of time.   }\label{fig:multirandomtimes}
\end{figure}

\begin{enumerate}
\item {\it Double-threshold phenomena}: We study the case when one urban center is the only influence for all other centers.
For example, we can think of this center as being Paris, which undoubtedly has an influence on all of France.
In order to explore the effect that
the intensity of the triggering event has on the spread of the rioting activity, we perform a series of simulations with
fixed the parameters $\omega = .2,\theta =.3,z_0 = 10, \beta =1, a = 100, p = .7, \eta =.2$
and vary the strength of the triggering event $A$.  In this case, we observe a {\it double threshold
phenomenon}, which we summarize below:
\begin{itemize}
\item[(a)] For $A$ small (the simulation illustrates the case when $A = 2$) we observe that the rioting activity does not spread beyond the location of where
the triggering event occurs.  In fact, the intensity is not sufficiently high to provoke the spatial spread of rioting activity.
Refer to Figures \ref{fig:2-1a}-\ref{fig:2-1c} for three
different snapshots in time of this simulation.  Red corresponds to high levels of activity and dark blue to zero levels of rioting activity.
\item[(b)] For $A$ of intermediate value (the simulation illustrates the case when $A = 6$) we observe that the rioting activity spreads in a local fashion: the burst of rioting activity affects nearest neighbors.
It appears that the local diffusion of $\lambda(s,t)$ dominates in this regime, refer to Figures \ref{fig:2-2a}-\ref{fig:2-2c} for the corresponding snapshots in time of this simulation.
\item[(c)] For $A$ large (the simulations illustrate the case when $A=10$) the spread of rioting activity is non-local:  the bursts of social activity jump infecting urban areas in a non-local manner.
It appears that the non-local diffusion of $\alpha(s,t)$ dominates in this regime, refer to Figures \ref{fig:2-3a}-\ref{fig:2-3c}.
\end{itemize}

The apparent threshold between the very localized rioting activity and the activity which spreads is not surprising, but the second observed threshold
is.  Indeed,  from these numerical experiments we {\it conjecture} that there is a critical threshold value for the intensity, denote it by $A^*,$ such that if
$A<A^*$ then the rioting activity spreads locally (the geographic network dominates in this regime) and if $A>A^*$ the rioting activity spreads non-locally
(the social-network dominates in this regime).

\item {\it Delay effect}: Consider a social network made up of two influential urban centers,
for example, this could be Marseilles (M) and Paris (P).  We assume that the triggering event occurs at time $t=0$
in (P).  If the initial triggering event is not sufficiently strong, we observe a nominal spread of the rioting
activity and in fact the activity ceases relatively fast.  Refer to Figure \ref{fig:4a}-\ref{fig:4b} for an illustration.
On the contrary, when a second exogenous event (even of significantly less strength) occurs at the second urban center (M)
at a much later time ($t=30$ for our simulation), we observe a much stronger and faster spread of activity both from
(M) and (P).  After the second exogenous event one observes a local spread of rioting activity from both (M) and (P).
Refer to \ref{fig:5a}-\ref{fig:5b} for an illustration of this phenomena.  Of course,
this is due to the self-reinforcement between the two center (P) and (M), which leads to a delayed burst of activity.
\end{enumerate}

\section{Spatially continuous deterministic model}
\label{sec:continuous}
While we can learn much from the discrete system and its numerical simulations, it is useful to work with the continuum limit in order
to be able to perform more rigorous analysis.  Although we recognize that a discrete model is more realistic for this application,
the continuum models are useful in that they allow us to determine the intrinsic parameters of the system  as well as to obtain estimates
for the duration of an episode of rioting activity.  Thus, we are motivated to introduce a couple of continuum versions of our models in this section.

\subsection{Local influence on the the social tension}
\label{sec:cont-local}
First, we consider the case when the diffusion of both the level of rioting activity and of the social tension is only due to geographic connections.
This corresponds to the case when the matrix $C = V.$
Let $\Omega\subset \Real^2$ (or $\Real^1$ as a first step in exploring this phenomena) represent our domain of interest.
We are interested in the limit as the number of nodes approaches infinity ($N\rightarrow \infty$): $\la(x,t)$ and $\alpha(x,t)$ are then defined for all $x\in \Omega.$
We
discretize $\Omega$ and obtain a lattice and denote it by $\Omega_l:=\set{(x_i,y_j)}_{i,j=1}^{N},$ such that $\Delta x =x_{i+1} -x_{i}$ and $\Delta y = y_{i+1}-y_{i}.$
The nodes $(x_i,y_j)\in \Omega_l = \N$ correspond to the discrete locations where the rioting activity will take place, e.g. `urban clusters' in the model
introduced in \cite{Braha2012}.  We are interested in the limit of the distance between the urban clusters approaching zero, i.e. $\Delta x\rightarrow 0$.
For simplicity, we consider the case when the stochastic effects are negligible, which corresponds to the system given by
equations \eqref{eq:la_nl} and \eqref{eq:al_nl}.
As the derivation of the continuum limit as both $\Delta x$ and $\Delta t$ go to zero is standard, we only mention that one uses
the discrete Laplacian operator in order to keep track of mesh distance and take the limit in such a way that
$D:= \frac{(\Delta x)^2 \eta}{\Delta t}$ remains constant.  Then, the corresponding limiting equation are the following:
\begin{subequations}\label{sys:cont1}
\begin{align}
&\frac{d}{dt}\la(x,t)= D\Delta\la(x,t) +r (\alpha(s,t)) G(\la(x,t)) - \kappa \la(x,t),\\
&\frac{d}{dt}\alpha(x,t)= D\Delta \alpha(x,t) + A_i\sum_{i=1}^n \delta_{t=0,s=s_i} -  (h(\la(x,t)) - \eta) \alpha(x,t)+\theta \alpha_b,
\end{align}
\end{subequations}
where $\kappa = \omega -\eta.$  We always assume that $\kappa>0$.  The
existence of solutions to the Cauchy problem defined by \eqref{sys:cont1} falls under
classical theory.

\subsection{Single triggering event: Eventual decay of the rioting activity}
\label{sec:cont-decay}
Let us consider the system \eqref{sys:cont1} with the following initial conditions:
\begin{align}\label{def:ic}
\la(x,0)=\lambda_0\quad\text{and}\quad\alpha(x,0)=0.
\end{align}
We show that in the case when $\alpha_b = 0$ the bursts of social activity eventually cease though out the domain
$\Omega$.  That is, the we prove that the mass of the social tension and the level of rioting activity eventually approaches zero.

\begin{prop}[Decay of mass]
Let $\la(x,t)$ and $\alpha(x,t)$ be solutions to \eqref{sys:cont1} with initial conditions given by \eqref{def:ic} for parameters such that
\[
\theta/(1+\lambda^*)^p>\eta.
\]
There exists $k_1,k_2>0$ such that
\begin{align}
&\norm{\alpha(0,x)}_1e^{-k_1 t}+A_i\sum_{i=1}^ne^{-k_1(t-t_i)}\nonumber\\ \le &\norm{\alpha(t,x)}_1 \le \norm{\alpha(0,x)}_1 e^{-k_2 t}+A_i\sum_{i=1}^ne^{-k_2(t-t_i)}.
\end{align}
Furthermore, for any $\epsilon>0$ there exists a $T_\epsilon>0$ such that $\int_\Omega\la(t)<\epsilon$ for all $t>T_\epsilon$.
\end{prop}

\begin{proof}
Let us first compute the dynamics of the mass of $\alpha,$
\begin{align*}
 \frac{d}{dt}\int_\Omega \alpha(x,t)\;dx = A_i\sum_{i=1}^n \int_\Omega\delta_{t=0,x=x_i}\;dx -  (h(\la(x,t)) - \eta) \int_\Omega\alpha(x,t)\;dx.
\end{align*}
For simplicity let us denote $y(t) = \int_\Omega\alpha(x,t)\;dx$  and $k_2 = \theta/(1+\lambda^*)^p-\eta$ and $k_1 = \theta-\eta.$
Then we obtain the following upper and lower bounds:
\begin{align*}
A_i\sum_{i=1}^n \int_\Omega\delta_{t=0,x=x_i}\;dx -  k_1 y(t)
\leq  \frac{d}{dt}y(t) \le A_i\sum_{i=1}^n \int_\Omega\delta_{t=0,x=x_i}\;dx -  k_2 y(t).
\end{align*}
This gives the upper bound on the mass:
\begin{align*}
y(t) \le &\frac{1}{e^{k_2t}} \int_0^t e^{k_2s}A_i \sum_{i=1}^n\delta_{s-t_i}\;ds\\
\leq &\norm{\alpha(0,x)}_1e^{-k_2t}+A_i\sum_{i=1}^ne^{-k_2(t-t_i)}.
\end{align*}
Similarly, we obtain that
\begin{align*}
y(t) \ge \norm{\alpha(0,x)}_1e^{-k_1 t}+A_i\sum_{i=1}^n e^{-k_1(t-t_i)}.
\end{align*}
Now, let us consider the dynamics of the mass of $\lambda:$
\begin{align*}
\frac{d}{dt}\int_\Omega\la(x,t)\;dx&= \int_\Omega r (\alpha)G(\la)\;dx - \kappa \int_\Omega\la(x,t)\;dx\\
& \le r (g(t))\abs{\Omega}\la^* - \kappa \int_\Omega\la(x,t)\;dx,
\end{align*}
where,
\[
g(t) = \norm{\alpha(0,x)}_1e^{-k_2t}+A_i\sum_{i=1}^n e^{-k_2(t-t_i)}.
\]
From this, we obtain the estimate
\[
\norm{\la(x,t)}_1 = \norm{\la(x,0)}_1e^{-\kappa t} +\tilde c \int_0^tr(g(s))e^{-\kappa s}\;ds.
\]
Note that $\lim_{x\rightarrow \infty} g(t)=0$ and this implies that $\lim_{t\rightarrow \infty}r(g(t))=0$. Thus, for any
$\epsilon>0$
there exists a time $T^*=T^*(t_i,A_i,k_2,a,\beta,\epsilon)$ such that for all $t>T^*$ $\la(t,x)<\epsilon$.
\end{proof}

\subsection{Non-local influence on the tension}
\label{sec:cont-nonlocal}
Similarly, we can obtain a continuum model for the case when the social tension diffuses
non-locally through the communications network $C$.  In this case, we obtain a system with
an integral operator.
\begin{subequations}\label{sys:nl}
\begin{align}
& \la_t = D\Delta \la - (\omega-\eta) \la +r(\alpha)G(\la)+\omega \la_b\\
& \alpha_t =\frac{\bar\eta}{\int_\Omega\J(\cdot,y)\;dy} \int_\Omega\J(x,y)\alpha(y,t)\;dy - h(\la)\alpha  +\theta \alpha_b+A\delta_{0,\bar s}.
\end{align}
\end{subequations}

The interaction potential $\J(x,y)$ is either equal to zero or one, which is due to the choice of the social network (see the
definition of the matrix C in section \ref{sec:discrete-model})
This, of course, can be generalized to include different weights.  As an interpretation of $\J(x,y)$
we can think that each location, $x\in \Omega$, has a
{\it domain of influence}, which is the set of locations which influence what happens in $x$.
At the same time, each location influences some locations and not others, we refer to this as the {\it range of influence} of location $x$.
Observe that the integral operator does not model diffusion and does not have to be symmetric.

\section{Traveling waves solutions}
\label{sec:waves}
One of the most important characteristics about the spread of rioting activity is the speed 
and the exact manner in which it spread.  In fact, 
one can think of the spread of rioting activity as a {\it front of high levels activity} that is 
invading regions with the base level of activity.  It is therefore natural to first look for the existence of traveling wave solutions.   

\subsection{Local influence}
For simplicity, let us consider the model with local diffusion for both variables and with only one exogenous event.  In fact, the shock
 can be included in the initial condition.  For this purpose, consider the system: 
\begin{subequations}\label{sys:ibvp_tw}
\begin{align}
&\frac{d}{dt}\la(x,t)= \Delta\la(x,t) +\Phi(\alpha(x,t),\lambda(x,t)),\\
&\frac{d}{dt}\alpha(x,t)= \Delta \alpha(x,t)+\Psi(\alpha(x,t),\lambda(x,t)),\\
&\la(x,0)=\lambda_0 \quad\text{and}\quad\alpha(x,0)=A_i\delta_{x=\bar x},
\end{align}
\end{subequations}
where $\Phi(\alpha,\lambda) : =r (\alpha) G(\la) - \kappa \la$ and $\Psi(\alpha,\lambda) :=-  h(\la)\alpha + \eta \alpha+\theta \alpha_b$.
Solving $d\alpha/dt = 0$ gives the relationship:
\begin{align}\label{eq:alss}
\alpha(\la) = \frac{\theta \alpha_b}{h(\la)-\eta}.
\end{align}

\begin{figure}[H] 
  \center
  \subfloat[$A=2,\;t\sim 4.3$]{\label{fig:2-1a}\includegraphics[width=0.31\textwidth]{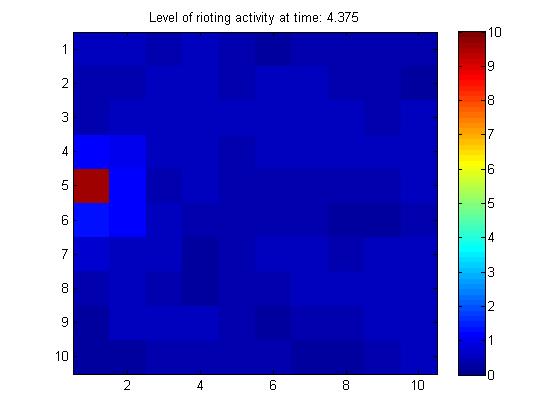}}\quad
  \subfloat[$A=2,\;t\sim 25$]{\label{fig:2-1b}\includegraphics[width=0.31\textwidth]{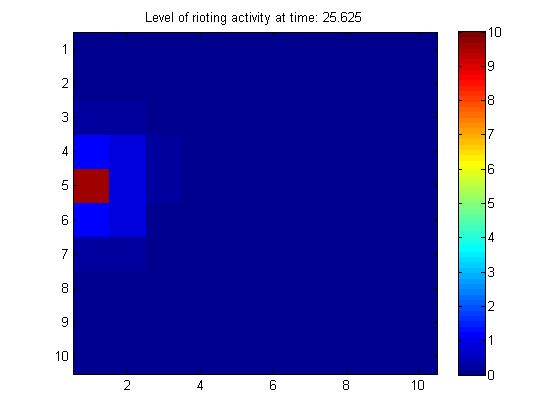}}
  \subfloat[$A=2,\;t= 50$]{\label{fig:2-1c}\includegraphics[width=0.31\textwidth]{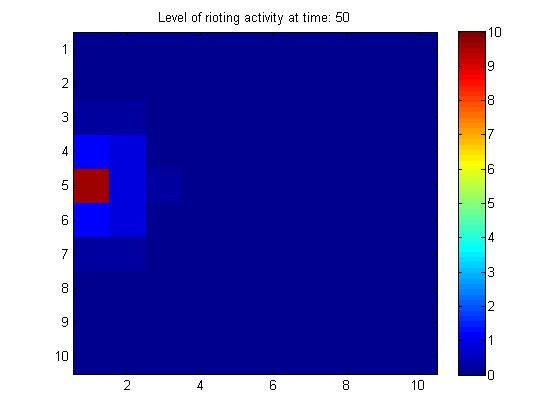}}\;
  \subfloat[$A=6,\;t\sim 10$]{\label{fig:2-2a}\includegraphics[width=0.31\textwidth]{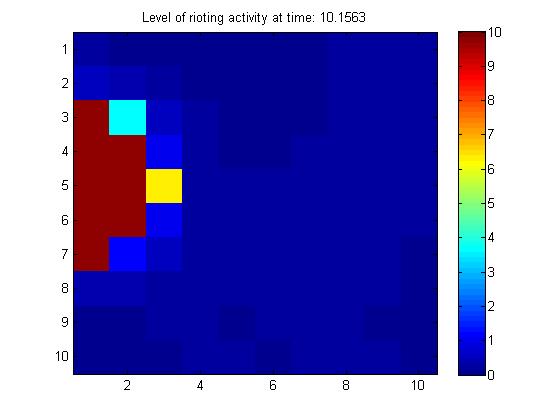}}\quad
  \subfloat[$A=6,\;t\sim 15$]{\label{fig:2-2b}\includegraphics[width=0.31\textwidth]{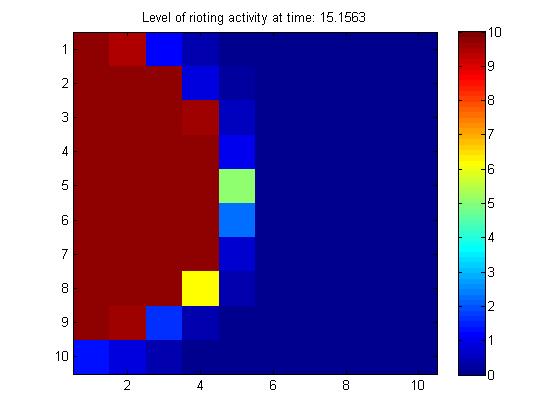}}
  \subfloat[$A=6,\;t= 50$]{\label{fig:2-2c}\includegraphics[width=0.31\textwidth]{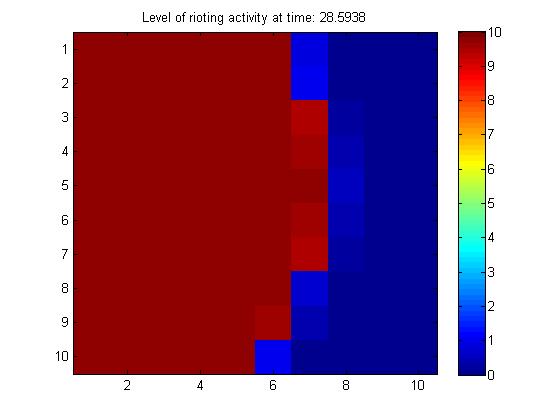}}\;
  \subfloat[$A=10,\;t\sim 4.3$]{\label{fig:2-3a}\includegraphics[width=0.31\textwidth]{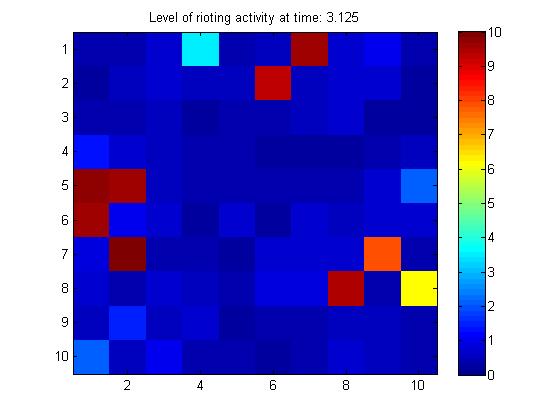}}\quad
  \subfloat[$A=10,\;t\sim 25$]{\label{fig:2-3b}\includegraphics[width=0.31\textwidth]{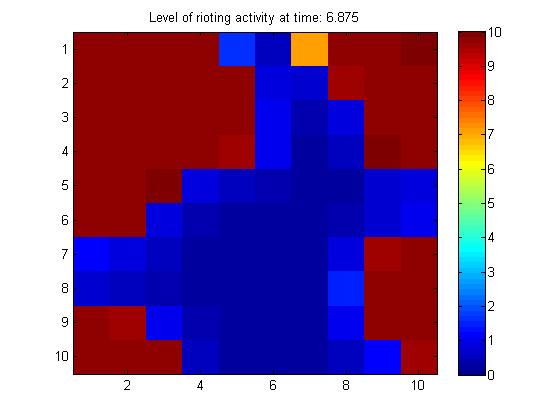}}
  \subfloat[$A=10,\;t= 50$]{\label{fig:2-3c}\includegraphics[width=0.31\textwidth]{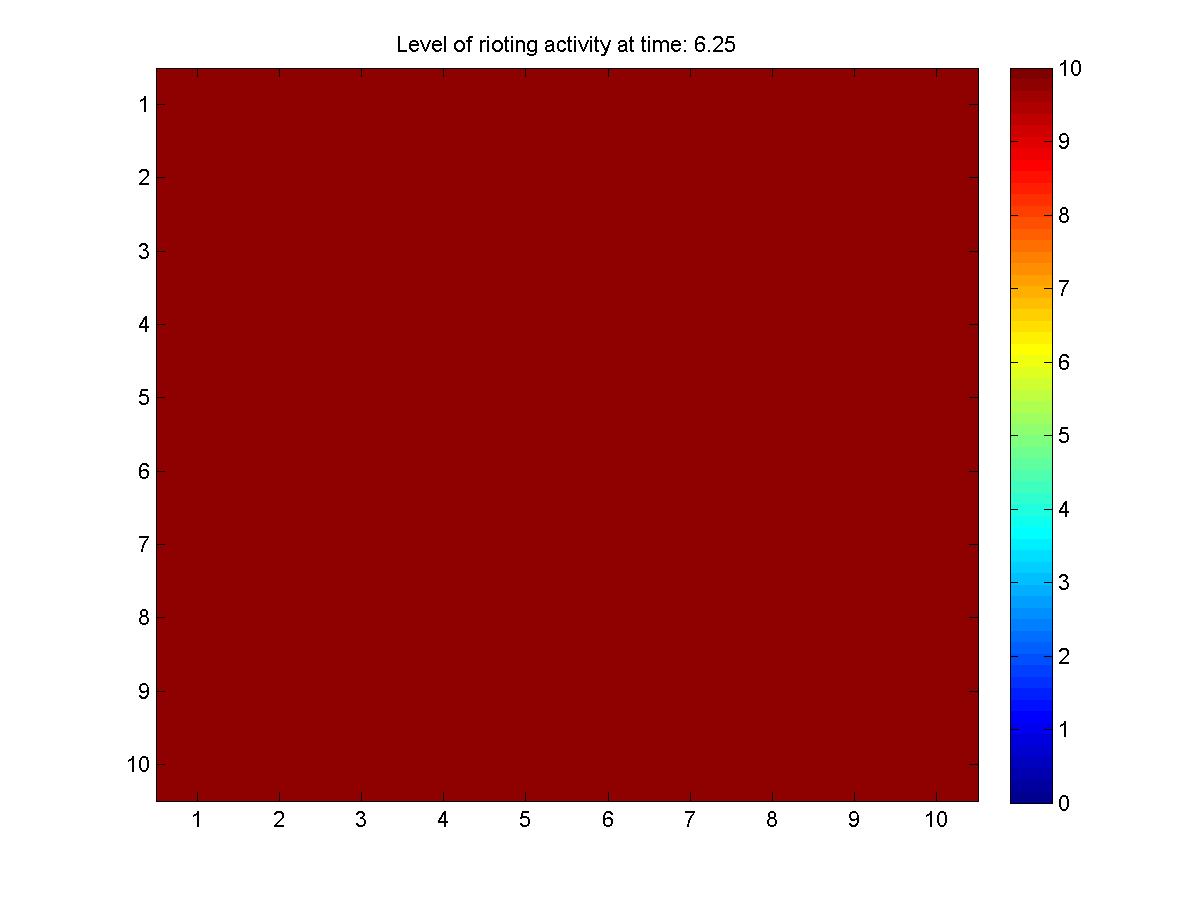}}\;
 \caption{Simulations of the the network-dependent model (system given by \eqref{eq:la_nl}-\eqref{eq:al_nl})
with $\omega = .2,\theta =.3,z_0 = 10, a = 100, p = .7, \eta =.2$ and varying intensity of the triggering event. Figures \ref{fig:2-1a}-\ref{fig:2-1c} 
illustrate a simulation with a triggering event occurring at time zero with intensity $A=2.$  On the other hand, Figures \ref{fig:2-2a}-\ref{fig:2-2c}
illustrate a simulation where the triggering event has intensity $A=4.$ Figures \ref{fig:2-3a}-\ref{fig:2-3c} illustrate a simulation with an triggering event of intensity $A=10$.}
\label{fig:2}
\end{figure}

\begin{figure}[H] 
  \center
  \subfloat[$t =25$]{\label{fig:4a}\includegraphics[width=0.4\textwidth]{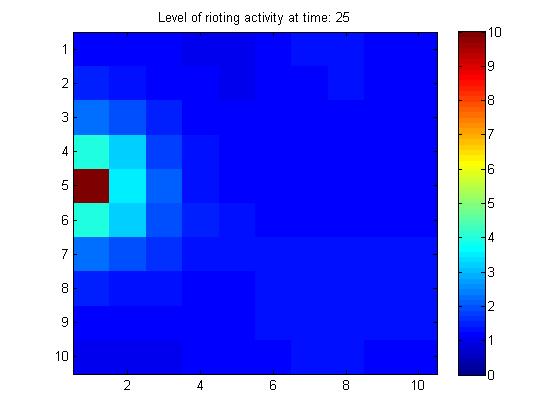}}
  \subfloat[$t =50$.]{\label{fig:4b}\includegraphics[width=0.4\textwidth]{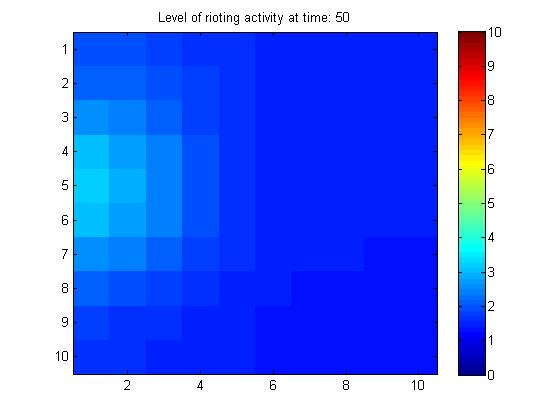}}\\
  \subfloat[$t =19$.]{\label{fig:5a}\includegraphics[width=0.4\textwidth]{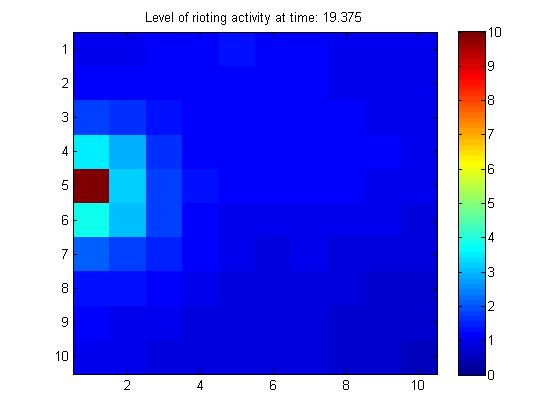}}
  \subfloat[$t =45$.]{\label{fig:5b}\includegraphics[width=0.4\textwidth]{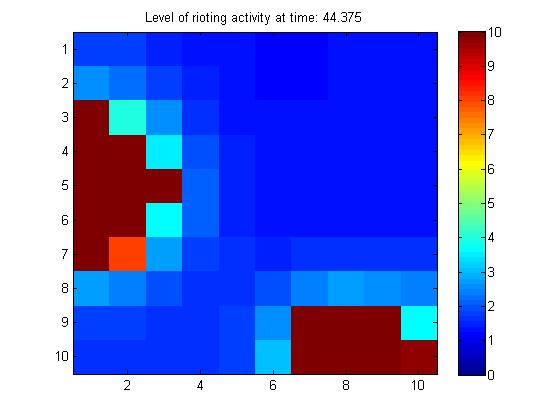}}
 \caption{Simulation with the social network with two influential nodes. Figures
\ref{fig:4a}-\ref{fig:4b} illustrate a single exogenous event with intensity $A =5.$  Figures \ref{fig:5a}-\ref{fig:5b}
illustrate the effect of a strong triggering event in (P) (red square in Fig \ref{fig:5a}) and a much waker second event $A=2$ 
at $t=30$, which leads to a delay spread in the rioting activity in (P).}\label{fig:3}
\end{figure}

Substituting \eqref{eq:alss} into the equation $\Phi(\alpha,\la)=0$ then gives two steady state solutions: one corresponding to a {\it non-excited state},
$(\alpha_1,\la_1) : =\left( \frac{\theta \alpha_b}{\theta-\eta},0\right),$ and one corresponding to an {\it excited state} $(\alpha_2,\la_2):=\left(\frac{\theta \alpha_b}{h(\la_2)-\eta},\la_2\right),$
where $\la_2=\la_2(G,r).$  The excited state will always be stable for our choice of functions $r(z)$ and $G(z).$  However, if
\begin{align}\label{cond:1}
r(\alpha_1)G'(\la_1) +\eta > h(\la_1)+k,
\end{align}
then the non-excited state will be unstable.  In fact, as the {\it critical tension} decreases (with all other parameters fixed) the system defined  by \eqref{sys:ibvp_tw} goes from a
regime where the two constant steady-states are stable into a regime where the non-excited steady-state becomes unstable.  Refer to Figure \ref{fig:6} for an illustration of the reaction terms 
$\Phi(\alpha(\lambda),\lambda)$ in the bistable regime (Figure \ref{fig:bis}) and monostable regime (Figure \ref{fig:mon}). 

In fact, for a certain parameter regime there exist traveling wave solutions, $(\psi,\phi,c)$ with $c\in \Real^+$, such that for all $z=x-ct$ the following holds:
\begin{align}\label{sys:traveling_wave}
\left\{\begin{array}{l}
\psi''(z) + c\psi'(z) +\Phi(\phi,\psi)=0, \\
c\phi'(z) + c\phi(z)+ \Psi(\phi,\psi) = 0,\\
0\leq (\psi(z),\phi(z))\leq (\alpha_2,\la_2),\\
\psi(-\infty) = \alpha_2,\;\phi(-\infty) = \la_2,\;\psi(+\infty) = \alpha_1,\;\phi(+\infty) = 0.
\end{array}\right.
\end{align}
The interesting observation is that the {\it critical tension} parameter $a$ can determine if traveling wave solutions exist
and whether they are unique.

\begin{figure}[H] 
\center
\subfloat[Bistable system: $a=5$]{\label{fig:bis}\includegraphics[width=0.45\textwidth]{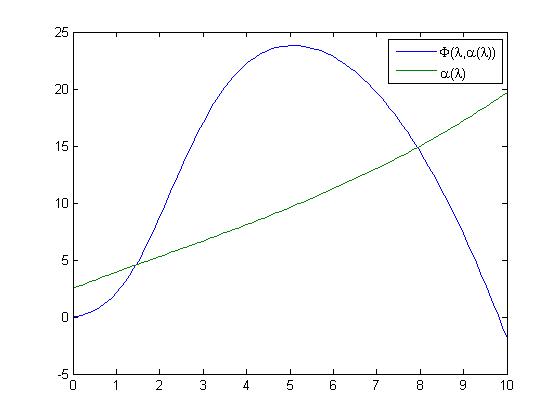}}\quad\quad
\subfloat[Monostable system: $a=1$]{\label{fig:mon}\includegraphics[width=0.45\textwidth]{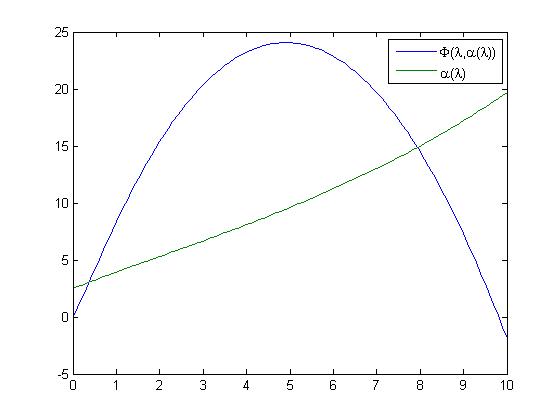}}
\caption{Illustration of the $\Phi(\alpha(\lambda),\lambda)$ with parameters: $z_0 = 10, \omega = .2, \theta=0.05,
\eta = .01, p=.5.$ }\label{fig:6}
\end{figure}

\begin{thm}[Critical threshold for traveling wave solutions]\label{thm:traveling_wave}
Let the $\kappa, \theta, p, \eta, a, k$ be chosen so that \eqref{sys:ibvp_tw} has two steady-state solutions $(\alpha_1,0)$ and
$(\alpha_2, \lambda_2)$ and such that \eqref{eq:alss} remains positive for all $\la\in [0,\la_2]$.  
There exists $0<a_*<\infty$ so that:
\begin{enumerate}
\item if $a_*<a$ there exists a unique $(\psi(z),\phi(z),c^*)$ with $c^*\in \Real$, up to translations, to \eqref{sys:traveling_wave}, 
Furthermore, $c^*>0.$
\item For $0<a<a_*$ there exists a $c^*\in \Real^+$ and
solutions $(\psi_c(z),\phi_c(z))$ with $z=x-ct$ to \eqref{sys:traveling_wave} for any $c\geq c^*$. 
Furthermore, when $c<c^*$, such waves do not exist.
\end{enumerate}
In either case, it holds that
$\psi'(z)< 0$ and $\phi'(z)<0$ for all $\abs{z}<\infty$.  
\end{thm}

The proof of Theorem \ref{thm:traveling_wave} is a direct application of two theorems from \cite{Volpert1994}.  
The only observation needed to be made is that for fixed parameters $\omega, \theta, \eta, p, z_0$ and $k$ as 
hypothesized in Theorem \eqref{thm:traveling_wave}, there exists a critical value $a^*$ such that if $a>a^*$ the \eqref{sys:ibvp_tw}
has two stable steady-state solutions.  
Thus, an application of Theorem 2.1 give part {\it 1}.  On the other hand, 
for $0\le a<a^*$ the non-excited state is unstable and an application of Theorem 2.2 gives {\it 2}.

This result shows that certain initial data or triggering events can facilitate the propagation of rioting activity.  
Moreover, we observe a {\it critical threshold} develop for the critical social tension $a$.  That is, for moderate levels of $a$
the traveling wave can only move at a unique speed.  However, if the a is sufficiently small then 
the traveling waves which facilitate the dispersal of rioting activity can spread
at very large speeds.  For example, a very intense triggering event could lead to a faster spread of the riots.    
We leave the study of the various qualitative properties of the traveling wave solutions from Theorem \ref{thm:traveling_wave}, such as the 
asymptotic decay rates, for future work.

\label{sec:waves-local}

\subsection{Numerical experiments on spreading}
\label{sec:waves-simul}
The existence of traveling wave solutions can be useful in exploring the spreading speed of the riots and 
helps provide a global picture.  However, the actual spread of the rioting activity in the 2005 French data
seems to have some very interesting properties, which would lead us to believe that the true solutions are
not true traveling waves, although they do exhibit similar characteristics.  To better explain this, let us denote by
$p(x)$ the peak number of events at location $x$ and by $t(x)$ the day in which the rioting activity peaked. 
In preliminary analysis of the spread of rioting activity during the 2005 French riots, we observe the following general trend:
if the triggering event occurred in location $x'$ at time $t=0$ then the peak number of events decreased and was delayed in locations 
with the distance away from the rioting activity, in other words, $p(y)$ is a decreasing function of $\abs{y-x'}$ and $t(y)$
is an increasing function of $\abs{y-x'}$.  This was only a general global trend and there were exceptions of course.  In particular,
large cities erupted first in some cases.  We are able to reproduce this observation with the model \eqref{sys:cont1}.
\begin{enumerate}
\item {\it Traveling wave-like solution}:  Figures \ref{fig:tw1}-\ref{fig:tw4} illustrate
the numerical solutions where $\lambda(x,0) = e^{-10 x}$ and a triggering event occurring at location $x=0$. The result is
a wave-like solution whose crest decreases as it moves.  In Figure \ref{fig:tw1} one observes that the rioting activity
grows and at time $t=1$ the solution has a wave-like profile which connects the maximum level of rioting activity 
and the base level of rioting activity.  At time $t=2$ the solution continues to have a wave-like profile, however, 
it now connects a lower level of rioting activity to the zero level of rioting activity.  Of course, this is not a true traveling
wave as we see the crest decrease and the transition between the lowest level of activity and the highest level
widen.  We extract and illustrate from the evolution of this solution, the level of rioting activity at four different
locations ($x=0,1,2,3$) in Figure \ref{fig:pp1}.  Recall, that the triggering event occurred in location $x=0$
and here the peak of the level of activity is the highest of the four and occurs the earliest.  From there
we observe that the peaks at each corresponding location decreases with the distance from the triggering event location.

 \item {\it Spreading solution}:  Figures \ref{fig:tp1}-\ref{fig:tp4} illustrate
the numerical solutions where $\lambda(x,0) = \la_0=2$ and a triggering event occurring at location $x=5$. The result is
a bump-like solution whose center is located at the triggering event location.  Initially, this bump-like solution increases.  
In Figure \ref{fig:tp1} one observed that the rioting activity
grows and at time $t=1$ it begins to simultaneously  spread and decay.
We extract and illustrate from the evolution of this solution, the level of rioting activity at four different
locations ($x=2,3,4,5$) in Figure \ref{fig:pp2}.  Recall, that the triggering event occurred in location $x=0$
and here the peak of the level of activity is the highest of the four and occurs the earliest.  From there
we observe that the peaks at each corresponding location decreases with the distance from the triggering event location.
\end{enumerate}
\subsection{Mixed local-non-local model}
\label{sec:waves-disc}

The inclusion of social connections in the network leads to a (spatially) non-local spread of information in the social tension.  
To the authors' knowledge this poses a new mathematical system for which there is much theory that needs to be developed. 
For example, while the global existence of solutions to the Cauchy problem defined by system \eqref{sys:cont1} falls under classical theory, 
we are not aware of any proof of the global existence of solutions to \eqref{sys:nl} and hope to address this issue in the near future.  Moreover, 
in the special case where there is a symmetric influence between nodes and this influence deceases with geographic distance we obtain a
system of the following form: 
\begin{subequations}
\begin{align}\label{sys:nl_sym}
& \la_t = D\Delta \la - (\omega-\eta) \la +r(\alpha)G(\la)+\omega \la_b\\
& \alpha_t =\bar \eta \left(\int_\Omega\J(x-y)\alpha(y,t)\;dy -\alpha\right) - (h(\la)+\bar \eta)\alpha  +\theta \alpha_b+A\delta_{0,\bar s}.
\end{align}
\end{subequations}
The existence of traveling wave solutions in this system is an open problem.  Although we conjecture that 
such solutions exists.

\section{Conclusion}\label{sec:disc}
Our goal here is to introduce simple models for the dynamics of riots that include what we believe to be the 
essential ingredients to qualitatively capture the initiation, self-excitation, spreading, and relaxation of rioting activities.  Without
any of these factors, we believe that important qualitative information will be lost. 
In addition, we have introduced a way to model the social or global connections both on a network and on a continuum 
domain.  The former is useful for the purpose of fitting data and the latter to extract some characteristics of 
the growth, spreading, and decay of the rioting activities.  These models are mainly exploratory tools that
can be used to gain some insight into the effect that the different mechanisms have on the rioting activity.
  
One interesting outcome of this model is the {\it double threshold phenomena} that comes from the intensity
of the triggering event.  It was surprising to us that if the intensity of an event was sufficiently high the spreading of
riots changed from being a local one to a non-local one.  Of course, since this is an observation from numerical simulations it raises the question to verify this with rigorous analysis: it is an open problem to prove that such a threshold exists.  

While all rioting activities are similar in nature there are many contrasts between them as well - see for example \cite{Newburn2014}.
We aim in this paper to construct a model that has the flexibility to describe various types of bursts of rioting activity.  It is hoped that it could be used in the future to determine if, for example, a riot was exogenously driven or endogenously driven. This is another interesting direction of research that is worth pursuing. Likewise, the model allows for different possible regimes of influence of activity on the social tension. Here we have chosen to consider the effect when high activity slows down the relaxation of this tension. This is reflected in the choice $p>0$ for the function $h(\lambda)$. In other situations, a high activity may actually generate a rapid decay of the social tension as some observations seem to suggest. This leads to choose $p<0$ in this function. Also note that another possible choice for the evolution of the social tension field is
\begin{equation}
\frac{d}{dt}\alpha(s,t) =  \sum_{i=1}^nA_i\delta_{t=t_i,s=s_i} -   h(\lambda) ( \alpha(s,t) - \alpha_b(s)).
\end{equation}
We will discuss the case $p<0$ and this variant elsewhere.

\begin{figure}[H] 
\center
\subfloat[Level of rioting activity in time.]{\label{fig:pp1}\includegraphics[width=0.45\textwidth]{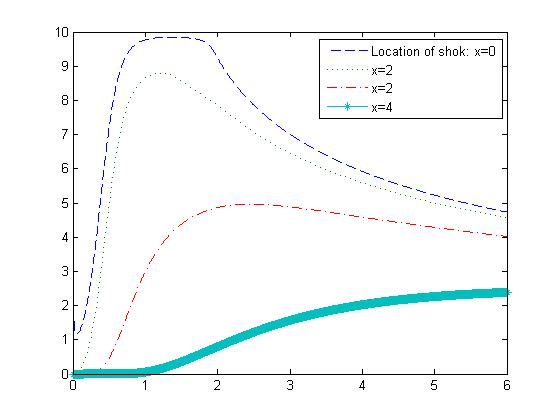}}
\subfloat[Level of rioting activity in time.]{\label{fig:pp2}\includegraphics[width=0.45\textwidth]{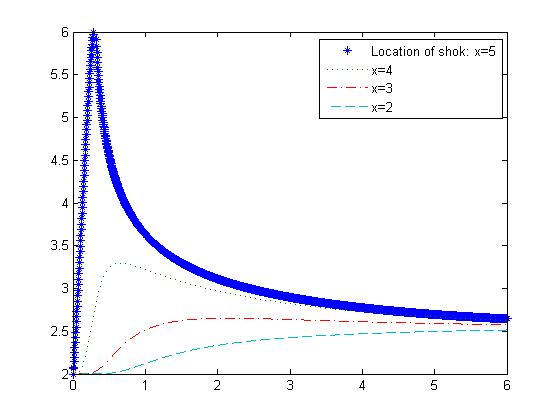}}
\caption{On the left we observe the total number of rioting activities for four different locations with parameters and initial conditions corresponding to
Figures \ref{fig:tw1}-\ref{fig:tw4}. On the left we observe the total number of rioting activities for four different locations with parameters and initial conditions corresponding to
Figures \ref{fig:tp1}-\ref{fig:tp4}.}\label{fig:7}
\end{figure}

Another interesting aspect of a riot that remains to be studied in detail is the process of ``self-relaxation."   
In fact, the 
reasons why riots end is recognized as quite unclear.  For example, in the 2005 French riots, riots were already strongly declining when the government adopted the strongest measures (notably an evening curfew). Rioters from the 2011 London riots reported to have stopped due to 
boredom, lack of targets, bad weather, fear of the police, and a call for peace by the father of a person killed during the riots \cite{Taylor2011}.  In the model, whenever there are no more shocks or other external inputs to the social tension, the rioting activity eventually ends (provided the background social tension is small enough). This is a key aspect of the model. The self-reinforcement effect leads to developing riots even though the social tension is already decreasing. However, when this social tension becomes small enough, it brings back the rioting activity to its baseline level. 

Some of the empirical factors mentioned in the introduction are clearly ingredients contributing to the decay of the social tension. It is well known that the weather plays a role in riots (everything else being equal, the likeliness of riots increases with the temperature). This factor can be thought of as modulating the baseline value of the social tension. Other empirical elements could be incorporated into our model. 
Specifically, the effects of the police department or announcements for the police that the army will become involved (such as what happened in the 2011 London riots) can be incorporated as a negative shock. This is obviously the same for the call for peace - but why a specific call, or such an announcement by the police, are actually perceived as events of strong amplitude is a deep issue in sociology, outside the scope of this modeling approach.  In this model, a negative shock on the social tension can bring the system into the regime of decreasing rioting activity. We would like to emphasize that a shock that would act directly on and only on the rioting activity would be inefficient - the system remaining in the regime where the activity increases. This bears some interesting consequences on what type of measures are more likely to ease a situation with severe unrests.
A more detailed study of the effect of negative shocks is also a natural direction of future work.

As already said, the purpose of this work is to establish models that characterize the global behavior of such systems.  However, the ultimate objective is to use data from 
riots to validate this model. One perspective is the use of the 2011 London riots data which are available on line. The case of the 2005 French riots would also  be very interesting to study: the fact that  
 the riots spread throughout the country makes this a particularly interesting data set to study the effects of non-local 
dispersal of the riots though the study of the proper social connections included in the model.  
In terms of stylized facts, our preliminary analysis seems to support that the model we propose here is coherent with the observations.

\begin{figure}[H] 
  \center
  \subfloat[{\footnotesize Solutions at $t=1.$}]{\label{fig:tw1}\includegraphics[width=0.45\textwidth]{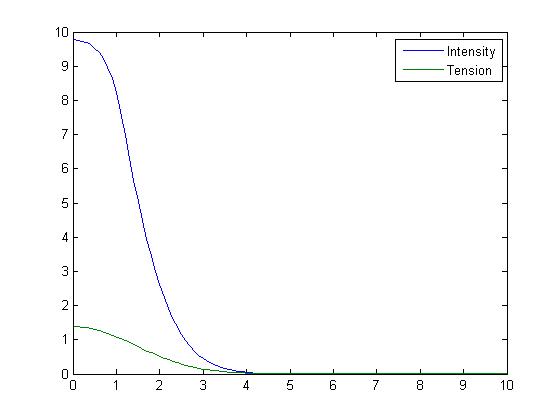}}
  \subfloat[Solutions at $t=2.$]{\label{fig:tw2}\includegraphics[width=0.45\textwidth]{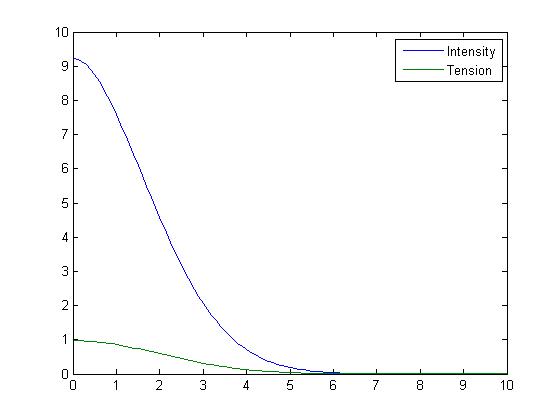}}\\
  \subfloat[Solution at $t=3.$]{\label{fig:tw3}\includegraphics[width=0.45\textwidth]{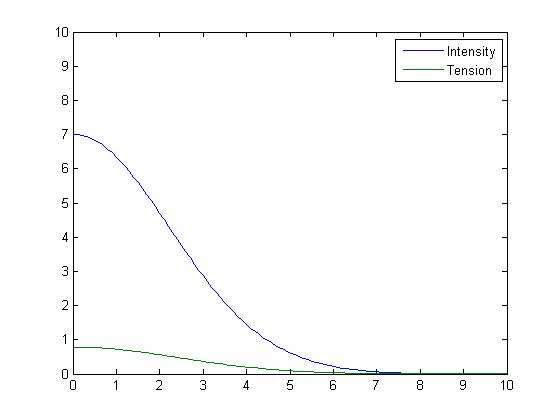}}
  \subfloat[Solution at $t=4.$]{\label{fig:tw4}\includegraphics[width=0.45\textwidth]{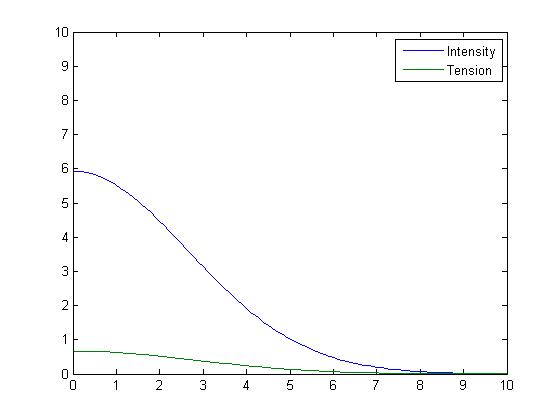}}\\
 \caption{Figures \ref{fig:tw1}-\ref{fig:tw4} illustrate the dynamics of the
numerical solution to system \eqref{sys:cont1} with exponential initial intensity at
in the triggering event location $x=0$ and with a triggering event intensity $A=50$.  
The horizontal axis represents space. }\label{fig:5}
\end{figure}

Finally, it is worth noting that modeling this social phenomena leads to the construction of 
new mathematical models, for which there is much mathematical theory to be developed. We hope that this article will encourage future work in this direction. In particular we emphasize the open problems related to the properties of the model with stochastic shocks and to the study of mixed local /non-local diffusion models. For the latter, solving the Cauchy problem, the construction of traveling fronts and the determination of the asymptotic speed of propagation are new types of problems which could lead to interesting mathematical developments.  Lastly, the different geometries of a network and heterogeneities in a system will lead to varying types of 
diffusion.  It would be an interesting problem to see how non-linear or fractional diffusion influence the the existence of traveling wave solutions and their propagation speed (if they exists).

\begin{figure}[H] 
  \center
  \subfloat[{\footnotesize Solutions at $t=.5.$}]{\label{fig:tp1}\includegraphics[width=0.45\textwidth]{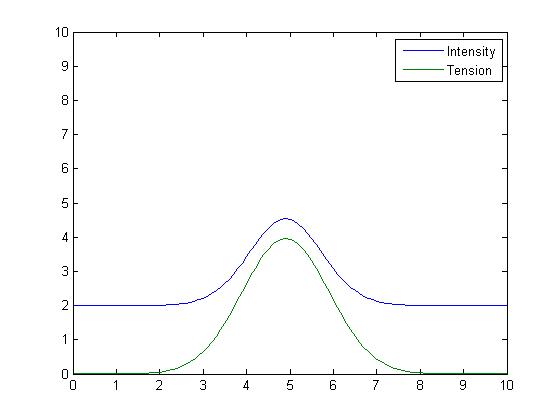}}
  \subfloat[Solutions at $t=1.5.$]{\label{fig:tp2}\includegraphics[width=0.45\textwidth]{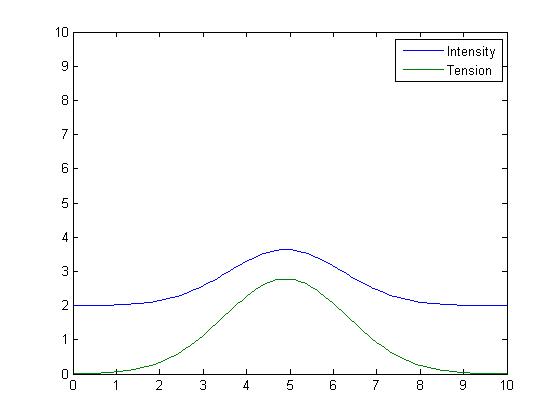}}\\
  \subfloat[Solution at $t=2.$]{\label{fig:tp3}\includegraphics[width=0.45\textwidth]{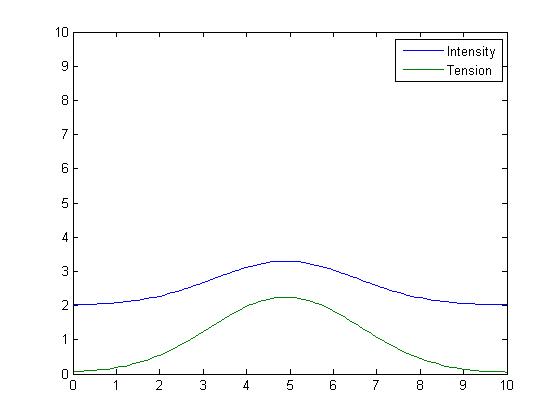}}
  \subfloat[Solution at $t=2.5.$]{\label{fig:tp4}\includegraphics[width=0.45\textwidth]{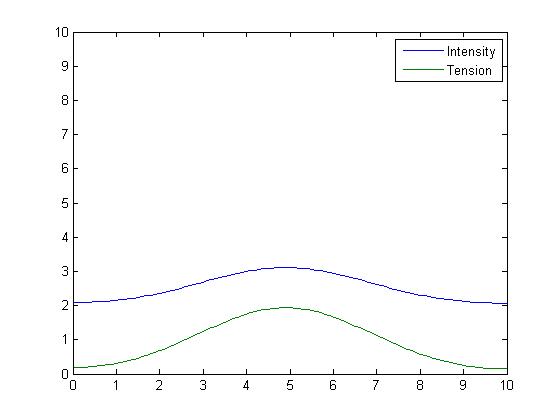}}
 \caption{
Figures \ref{fig:tp1}-\ref{fig:tp4} illustrate the dynamics of the
numerical solution to system \eqref{sys:cont1} with a constant initial conditions $\la(x,0)=2$
and a triggering event with intensity $A=100$ occurring at location $x=0$.  
The following parameters were used for both numerical experiments: $z_0 = 10,a = 100,\omega = .2,\theta=0.05,
\eta = .198.$ The horizontal axis represents space. }\label{fig:6b}
\end{figure}

\section*{Acknowledgments}
The research leading to these results has received funding from the European Research Council
under the European Union's Seventh Framework Programme (FP/2007-2013) / ERC Grant
Agreement n. 321186 - ReaDi - ``Reaction-Diffusion Equations, Propagation and Modelling'' held by Henri Berestycki. 
This work was also partially supported by the French National Research Agency (ANR), within  project NONLOCAL ANR-14-CE25-0013, and by the French National Center for Scientific Research (CNRS), within the project PAIX of the CNRS program PEPS HuMaIn.

The authors are grateful to Sebastian Roch\'e and Mirta Gordon for the discussions that motivated this work. They also thank Laurent Bonnasse-Gahot for discussions in the later stages of this work. 
\newpage
\bibliographystyle{alpha}
\bibliography{library}
\end{document}